\documentclass[12pt]{amsart}

\def\dist{{\rm dist}}
\def\eps{{\varepsilon}}

\def\mes{{\rm mes}}

\def\Card{{\rm Card}}

\def\Osc{\mathop{\rm Osc}}
\def\one{{\mathbf{1}}}
\def\Prob{{\mathbb{P}}}

\def\bbG{\mathbb{G}}

\def\bbL{\mathbb{L}}

\def\naturals{\mathbb{N}}

\def\Tor{\mathbb{T}}

\def\reals{\mathbb{R}}

\def\integers{\mathbb{Z}}

\def\RmII{{I\!\!I}}

\def\bG{\mathbf{G}}

\def\be{\mathbf{e}}
\def\bg{\mathbf{g}}

\def\bk{\mathbf{k}}

\def\bz{\mathbf{z}}

\def\brZ{{\bar Z}}

\def\brc{{\bar c}}

\def\brn{{\bar n}}

\def\brz{{\bar z}}

\def\brnu{{\bar \nu}}

\def\brtau{{\bar\tau}}

\def\brsigma{{\bar \sigma}}

\def\brPhi{{\bar \Phi}}

\def\brOmega{{\bar\Omega}}

\def\cA{\mathcal{A}}

\def\cB{\mathcal{B}}

\def\cC{\mathcal{C}}

\def\cD{\mathcal{D}}

\def\cJ{\mathcal{J}}

\def\cG{\mathcal{G}}

\def\cE{\mathcal{E}}

\def\cK{\mathcal{K}}

\def\cL{\mathcal{L}}

\def\cR{\mathcal{R}}

\def\cS{\mathcal{S}}

\def\cT{\mathcal{T}}

\def\cV{\mathcal{V}}

\def\cW{\mathcal{W}}

\def\cZ{\mathcal{Z}}

\def\fA{\mathfrak{A}}

\def\fB{\mathfrak{B}}

\def\fG{\mathfrak{G}}

\def\fg{\mathfrak{g}}

\def\fM{\mathfrak{M}}

\def\fm{\mathfrak{m}}

\def\fp{\mathfrak{p}}

\def\fq{\mathfrak{q}}

\def\ft{\mathfrak{t}}

\def\fz{\mathfrak{z}}

\def\hC{{\hat C}}

\def\hG{{\hat G}}

\def\hL{{\hat L}}

\def\heps{{\hat{\eps}}}

\def\hmu{{\hat\mu}}

\def\hnu{{\hat\nu}}

\def\hPhi{{\hat\Phi}}

\def\tC{{\tilde C}}

\def\tG{{\tilde G}}

\def\tfm{{\tilde\fm}}

\def\tR{{\tilde R}}

\def\tT{{\tilde T}}

\def\tX{{\tilde X}}

\def\tV{{\tilde V}}

\def\tW{{\tilde W}}

\def\tk{{\tilde k}}

\def\tn{{\tilde n}}

\def\tx{{\tilde x}}

\def\tz{{\tilde z}}

\def\tmu{{\tilde\mu}}

\def\tsigma{{\tilde\sigma}}

\def\ttau{{\tilde\tau}}

\makeatother

\def\beq{\begin{equation}}
\def\eeq{\end{equation}}

\usepackage{color}
\usepackage[usenames,dvipsnames,svgnames,table]{xcolor}
\usepackage[normalem]{ulem}

\newtheorem{theorem}{Theorem}[section]
\newtheorem{proposition}[theorem]{Proposition}
\newtheorem{lemma}[theorem]{Lemma}

\newtheorem{corollary}[theorem]{Corollary}

\theoremstyle{remark}

\theoremstyle{definition}
\newtheorem{defn}[theorem]{Definition}
\newtheorem{ex}[theorem]{Example}

\numberwithin{equation}{section}

\def\stackover#1#2{\mathrel{\mathop{#1}\limits_{#2}}}

\def\DS{\displaystyle}

\author{Dmitry Dolgopyat and P\'eter N\'andori}
\title{Infinite measure mixing for some mechanical systems}


\begin{document}
\maketitle

\begin{abstract}
We show that if an infinite measure preserving system is well approximated on most of the phase space
by a system satisfying the local limit theorem, then the original system enjoys mixing with respect to
global observables, that is, the observables which admit an infinite volume average.
The systems satisfying our conditions include the Lorentz gas with Coulomb potential, 
the Galton board and
piecewise smooth Fermi-Ulam pingpongs.
\end{abstract}

\section{Introduction}
\label{sec:int}

Mixing plays a central role in the study of stochastic properties of
dynamical systems preserving a finite measure. 
Recently, there has been a surge of interest in studying mixing properties of infinite measure preserving systems
(\cite{G11, MT12, MT13, M15, T15, AB16, CD16, LT16, T16, OP17, AN17, MT17, P17, DN17i, MT18, DNP19}). 
  Contrary to the case of finite measures, there are several different notions of mixing
in the infinite measure preserving case.

A driving force behind the development of ergodic theory and dynamical systems has always been a
desire to understand physical systems. That is why we study here the question
of infinite measure mixing for specific mechanical systems. 
In many such systems, it is natural to assume some periodicity or approximate periodicity
and to study the functions whose averages over large boxes stabilize.  
The notions
of global mixing introduced recently by Marco Lenci \cite{L10} (and further studied in \cite{L13, BGL17, L17})
are particularly suitable for our purposes. 

We will approximate our system by a periodic one: a $\mathbb Z^d$-extension of a map $f$ acting on a compact
space $M$ and preserving a finite measure. Many finite measure preserving mechanical systems $f$ 
are hyperbolic and enjoy good mixing properties, such as the local limit theorem (LLT).
It turns out that the notions of LLT and mixing of the extended system are nicely connected.
We have studied this connection (for different notions of mixing)
in our recent work \cite{DN17, DN17i}. By further exploiting this
relation, we are able to prove global mixing for several mechanical systems.

Next, we give informal definitions of the notions of global mixing.
Let $T$ be a map of a space $X$ preserving an infinite measure $\mu.$
The idea of \cite{L10} is to introduce two spaces: the space of local functions $L^1$
and the space of global functions $\bbG\subset L^\infty. $ The functions from $\bbG$ are supposed to
admit an average value 
$$ \brPhi=\lim_{\mu(V)\to\infty} \frac{1}{\mu(V)} \int_V \Phi d\mu $$
where the limit has to be understood in an appropriate sense.  
The map $T$ is called {\em local global mixing} if for each $\phi\in {L^1} (\mu)$ 
and each $\Phi\in \bbG$
we have
\begin{equation}
\label{LocGlobMix}
\lim_{n\to\infty} \int \phi(x) \Phi(T^n x) d\mu=\left(\int \phi d\mu \right)\; \brPhi. 
\end{equation}
$T$ is called {\em global global mixing} if for each $\Phi_1, \Phi_2\in \bbG$ for large $n$ and large $V$,
$$ \frac{1}{\mu(V)} \int_V \Phi_1(x) \Phi_2(T^n x) d\mu\approx \brPhi_1 \brPhi_2. $$

The rest of the paper consists of two parts: an abstract part and an applied part. 
In Section \ref{ScResults}, we define an abstract framework and formulate several results
implying local global and global global mixing for periodic or approximately periodic maps 
preserving an infinite invariant measure. In Section \ref{ScProofs}, we prove these results.
In Section \ref{ScFlows}, we extend the previous results to flows; still in an abstract framework.

The second part of the paper is about explicit examples where the abstract results can be applied.
In the preliminary Section
\ref{BilPrel} we review theory of hyperbolic dynamical systems with singularities. We focus on the Sinai 
billiards and related models.
The most important results of this paper are reported in Section \ref{secEx}. Here, we study local global and global global
mixing of several mechanical systems. Our examples include the following 
variants of Lorentz gas: periodic, locally perturbed, confined to a half strip,
subject to an asymptotically vanishing potential field and with Gaussian thermostats.
Besides the Lorentz gas, we study Galton boards, the Fermi Ulam pingpong and bouncing balls in a gravity field.
A reader interested in one of these examples can proceed to the appropriate subsection of Section \ref{secEx}
after reading the abstract part.
In some cases (in particular, the periodic ones) the application of the abstract results from the first part is straightforward.
  In other cases a significant amount of work is required to verify our abstract assumptions. This turns out to be most difficult in the case
of the Lorentz gas with asymptotically vanishing potential, and we present the most technical step of our analysis in the 
separate Section \ref{sec:m6}. 
{ We hope that a similar approach could be used to analyze other nonuniformly
hyperbolic mechanical systems.}
Section \ref{sec:m6} also ontains an important recurrence-transience dichotomy, which is of independent interest. 
Finally, we give a short summary of our results and mention some future research directions in Section \ref{ScConclusion}.

\section{Abstract results}
\label{ScResults}

{
\subsection{Periodic systems}
\label{ScResults1}

Let us start with periodic systems.}
Let $X=M\times\integers^d$,
$x = (y,z) \in X$ and $T(y,z)=(f(y), z+\tau(y))$ where $M$ is a {compact} 
metric space and
$f$ preserves a {Borel} probability measure $\nu.$
We equip $X$ with the measure $\mu$ which is the product of $\nu$ and the counting measure
on $\integers^d$. We write 
$$ \tau_n(y) =\sum_{j=0}^{n-1} \tau (f^{j}(y)).$$

{We now specify our choice of the space of global functions $\bbG$ 
to provide the rigorous definition of local-global and global-global mixing. In fact, we consider three
classes of global functions. \\

We say that $V \subset X$ is a {\em cube} if 
$V = M \times ( z + (-\lfloor w /2 \rfloor, w -\lfloor w /2 \rfloor]^d)$ for some $z \in \integers^d$ and $w \in \integers_+$. We also say that
$z$ is the center and $w$ is the size of the cube. }

\begin{defn}
Let 
$\bbG_O$ be the space of bounded uniformly continuous functions 
{$\Phi: X \to \mathbb R$ for which there exists
$\brPhi \in \reals$}
such that for {any 
$a_{1}, a_2 , \dots, a_d, b_1, b_2,...,b_d \in \mathbb R$
with $a_{i} < b_{i}$,}
$$ \lim_{N\to\infty} \frac{1}{\prod_j (b_{j}N - a_{j}N) } \int_
{{x=(y,z): z \in} \prod_j [a_{j} N, b_{j} N]} \Phi(x) d\mu(x)=\brPhi.$$
\\

{
Let
 $\bbG_U$ be the space of bounded uniformly continuous functions 
$\Phi: X \to \mathbb R$ for which there exists
$\brPhi \in \reals$}
such that for each 
$\eps$ there exists $N_0$
such that for each cube $V$ of size greater than $N_0$ we have
\begin{equation}
\label{eq:globalmixdef}
 \left|\frac{1}{\mu(V)} \int_{V} \Phi(x) d\mu(x)-\brPhi\right|\leq \eps.
\end{equation}
\\

{
We say $\Phi \in\bbG_{AO}$ if 
$\Phi$ is a uniformly continuous functions 
from $X$ to $\mathbb R$  for which there exists
$\brPhi \in \reals$
such that for every $\eps >0$ there exists 
$b(\eps) \in \integers_+$, and $B_0=B_0(\eps) \in \integers_+$ so that for all $B > B_0$ we have
$$|\cG_{b,B}| > (1 - \eps) B^d,$$ 
where $\cG_{b,B}$ denotes the set of points
$z \in ((-B/2,B/2]^d \cap \integers^d)$ so that the cube $V$ centered at $z$
and of size $b$ satisfies \eqref{eq:globalmixdef}.}
\end{defn}

{
We note that $\bbG_U \subset \bbG_{AO} \subset \bbG_O$
(the first containment is trivial, the second one follows from approximating a 
large
rectangular box by a disjoint union of smaller cubes). The notation "O"
represents that we require closeness to the average
on boxes containing the origin; "AO" represents approximate 
closeness to the average near the origin and "U" stand for uniform.
{  $\bbG_O$ is the largest space 
of global functions
where one could hope to obtain mixing while $\bbG_U$ is the smallest
space of interest. It turns out that $\bbG_O$ is too large for limit theorems,  see Example \ref{ExNoGlobMix}.
The intermediate space $\bbG_{AO}$ has better
properties since it captures the notion that the global observables are often "close to the local equilibrium
on mesoscopic scales" (which is represented by $b$ in our definition). An important class of global observable
are provided by {\em functions of a random environment}. Namely, let $h^z$ be an ergodic $\integers^d$ 
action on a
 space $\Omega$
preserving a measure $P.$ Given a function $\Psi$ on $M\times \Omega$ let
$\Phi_\omega(x,z)=\Psi(x, h^z \omega).$ Then it follows from the ergodic theorem 
that $\Phi_\omega\in \bbG_{AO}$ for $P$-a.e. $\omega$.
We refer the reader to \cite{DDKN} for the applications of these ideas to the study of mixing properties 
of skew products.
}\\

With the  definitions of $\bbG_O,\bbG_{AO}, \bbG_U$, \eqref{LocGlobMix} furnishes the definition of local-global mixing
with respect to $\bbG_O, \bbG_{AO},\bbG_U$. Next we define global-global mixing.}

\begin{defn}
\label{def:gg}
$T$ is global-global mixing {with respect to 
$\bbG_O/\bbG_{AO}/\bbG_U$} if for each  $\Phi_1, \Phi_2\in{\bbG_O/\bbG_{AO}/\bbG_U}$,
$$ \lim_{n\to\infty} \limsup_{{V \in \cV}, \mu(V)\to\infty} 
\frac{1}{\mu(V)} \int_V \Phi_1(x) \Phi_2(T^n x) d\mu=$$ 
$$\lim_{n\to\infty} \liminf_{{V \in \cV}, \mu(V)\to\infty} 
\frac{1}{\mu(V)} \int_V \Phi_1(x) \Phi_2(T^n x) d\mu= 
\brPhi_1 \brPhi_2. $$
{
Here, $\cV$ is the collection of cubes containing $M \times \{ 0 \}$ in case of $\bbG_O$ and
$\bbG_{AO}$
and the collection of all cubes in case of $\bbG_U$.}
\end{defn}

\begin{defn}
$T$ satisfies a {\em mixing local limit theorem (MLLT) at scale $L_n$} if there is a {bounded, continuous function
$\fp: \reals^d \to [0, \infty)$} such that
\begin{equation}
\label{P-Prob}
 \int \fp( z) d {{\mbox Leb}(z)}=1 
\end{equation}  
and for each $\phi_1, \phi_2\in C(M)$ for each $\mathbb Z^d$-valued sequence 
$z_n^0$ such that $z_n^0/L_n\to 0$ and for each $K < \infty$,  
\begin{equation}
\label{EqMixLLT}
 \lim_{n\to\infty} 
 \sup_{z \in \reals ^d, |z| < K} \left| 
 L_n^d \int \phi_1(y) \phi_2(f^n (y)) \one_{\tau_n= z_n^0 + \lfloor z L_n \rfloor} d\nu -
\nu(\phi_1) \nu(\phi_2) \fp(z) \right| =0
\end{equation}
where $\lfloor . \rfloor$ means taking lower integer part coordinate-wise.

$T$ satisfies a {\em shifted mixing local limit theorem at scale $L_n$} if there is a sequence $D_n\in \reals^d$ 
and a continuous function
$\fp$ satisfying \eqref{P-Prob},
such that for each $\phi_1, \phi_2\in C(M)$ for each $\mathbb Z^d$-valued sequence 
$z_n^0$ such that
$\dfrac{z_n^0-D_n}{L_n}\to 0,$ and for each $K < \infty$, \eqref{EqMixLLT} holds.
\end{defn}

We remark that  the MLLT implies the following useful {\em a priori} bound: if $\phi_1, \phi_2$ are 
{ bounded} functions, then 
$$ \left|\int \phi_1(y) \phi_2(f^n (y)) \one_{\tau_n= z_n^0 + \lfloor z L_n \rfloor} d\nu\right|
\leq 
C ||\phi_1||_\infty ||\phi_2||_\infty L_n^{-d}. $$
 
Now a standard approximation argument shows that
 the convergence in \eqref{EqMixLLT} is uniform for $\phi_1, \phi_2$ in a 
compact subset of $C(M)$ {(w.r.t. the $C^0$ topology).} The same remark applies to all variants of the MLLT considered
in this paper, i.e. to the shifted MLLT, the AMLLT and condition (M4) (the last two are to be defined later).

\begin{theorem}
\label{LLTIMMIx}
Suppose that $T$ satisfies MLLT. Then 

(a) $T$ is local global mixing with respect to $\bbG_O;$ 

(b) $T$ is global global mixing with respect to {$\bbG_{AO}.$}
\end{theorem}

For random walks, { part (a)} is proven in \cite{Br05}. The proof { of  
Theorem~\ref{LLTIMMIx}}
 follows the arguments of 
\cite{Br05}, however, we will provide the proof in \S \ref{SSCocycle} since our setting is quite different from
that of \cite{Br05}.

\begin{theorem}
\label{SLLTIMMIx}
Suppose that $T$ satisfies a shifted MLLT. Then 

(a) $T$ is local global mixing with respect to $\bbG_U;$ 

(b) $T$ is global global mixing with respect to $\bbG_U.$
\end{theorem}

{
In the remaining part of  
\S \ref{ScResults1},
we comment on the suitability of the spaces 
$\bbG_O, \bbG_{AO}, \bbG_U$ for our setup. First,
we note that $\bbG_O$ and $\bbG_{AO}$ are suitable spaces in case the MLLT holds with zero drift.
In case 
the shifted MLLT holds with non-zero drift, we need to work with the smaller space $\bbG_U$
as suggested by the following example.
}

{
\begin{ex}
\label{ExNoGlobMix}
Suppose that $d=1$,
$\tau$ is bounded
and the MLLT holds with $L_N=\sqrt{N}$ and a Gaussian $\fp$.
Let $\Phi(y,z)=(-1)^m$ if $m^3\leq |z|<(m+1)^3$ for some non-negative integer $m.$
One can easily check that
$\Phi\in \bbG_O$ and $\brPhi=0.$ On the other hand,
we claim that for each $N$,
\begin{equation}
\label{eq:exsec2}
\lim_{{V \in \cV}, \mu(V)\to\infty} 
\frac{1}{\mu(V)} \int_V \Phi(x) \Phi(T^N x) d\mu =1,
\end{equation}
where $\cV$ is the collection of boxes containing 
$M \times\{ 0\}$. \eqref{eq:exsec2} shows that
 global-global mixing with respect to $ \bbG_O$ does not hold.
To prove \eqref{eq:exsec2}, note that $\Phi(y,z) \Phi(T^N(y,z)) =1$ whenever 
$$m^{3} + N\| \tau \|_{\infty} < |z| < 
(m+1)^{3} - N\| \tau \|_{\infty}$$ for 
some non-negative integer $m$ and the relative measure
of such points $(y,z)$ in large boxes is close to $1$.

Next suppose that $T$ satisfies a shifted LLT with $D_N=vN$ for some $v>0$,
$L_N=\sqrt{N}$ and a Gaussian $\fp$. Let $\phi$ be a compactly supported Lipshitz 
probability density on $X$.
For any large positive integer $m$, there exists another large positive integer
$N$ so that 
\begin{equation}
\label{N-NolgMix}
 \left|D_N-\frac{(2m)^{3}+(2m+1)^{3}}{2}\right|\leq v. 
\end{equation}
 
Since $(2m+1)^{3} -  (2m)^{3} \asymp
m^{2} \gg
 m^{3/2} \asymp N^{1/2}$,
the LLT implies that $\Phi(T^N x)=1$ for most $x$ in the support of $\phi,$
and so
\begin{equation}
\label{SeeOne}
 \left| \int \phi(x) \Phi(T^N x) d\mu - 1 \right|  =
o_m(1).
\end{equation} 
 Consequently, $T$ does not satisfy
local global mixing with respect to~$\bbG_O.$

Next, set $m_j=2^j$ and let 
$$\Phi(y,z)=\begin{cases} 1 & \text{if }(2m_j)^3\leq z<(2m_j+1)^3 \text{ for some }j \\
0 & \text{otherwise.} \end{cases} $$
One can check that
$\Phi\in \bbG_{AO}$ with $\brPhi=0$, however, taking $N$ given by \eqref{N-NolgMix}
with $m=m_j$, we get \eqref{SeeOne} showing that the
local global mixing fails on $\bbG_{AO}$ as well.
\end{ex}

Example \ref{ExNoGlobMix} shows that $\bbG_O$ and $\bbG_{AO}$ are too large for global mixing in some cases.
A typical application of mixing is to control the ergodic sums. A more sophisticated version
of Example \ref{ExNoGlobMix} given in \cite{DLN19} 
shows that the Law of Large Numbers also fails on those spaces (at least in the context of random walks),
so one needs to consider smaller spaces. One can argue that the space $\bbG_U$ is too small for many
applications. To address this issue, \cite{DLN19} introduces larger spaces, where, in the context
of random walks, one can prove local global mixing and the Law of Large Numbers. However the spaces
from \cite{DLN19} involve some additional parameters, so using them would make the present work
significantly more complicated. 
We prefer to work on $\bbG_U$ in order to highlight the main
ideas of our approach.}
 
\subsection{Almost periodic systems}
\label{sec:AMLLT}

{
The main results of this paper concern systems that are 
close to periodic in some sense but not
exactly periodic.
Let us now consider a map $\tT$ acting on the space
$$
\tX = \left[ \cup_{z \in \cB} \left( D_z \times \{ z\} \right) \right] \cup
\left[ \cup_{z \in \integers_+^{d_1}\times \integers^{d_2} \setminus \cB } \left(M \times \{ z\} \right) \right]
$$
where 
$d_1$ and $d_2$ are
 non-negative integers, $M$ and $D_z$, $z \in \cB$ are compact metric spaces. 
This setup is more general than the one in \S \ref{ScResults1}: on the one hand
we allow $\mathbb Z_+$ in the
phase space to model systems with 
global reflections and on the other hand we allow some drastic departure from perodicity: whenever $z \in \cB$, 
the phase space $D_z$
can be different from $M$. 

We  assume that $\cB$ is small in the following sense. 
For every $\eta >0$ there is $\xi=\xi(\eta)$ and $Q_0=Q_0(\eta)$ 
 so that for $Q\geq Q_0$
\begin{equation}
\label{eq:cBcond}
\frac{\left|\left\{ 
\bk \in \left[0, Q \right]^{d_1} \times \left[- \frac{Q}{2}, \frac{Q}{2} \right]^{d_2} \cap 
\integers^{d_1 + d_2}: \dist(\bk, \cB)\leq \xi Q  
\right\}\right|
}{Q^{d_1 + d_2}} < \eta.
\end{equation}

Furthermore, we assume that $\tT$ preserves a $\sigma$-finite measure 
$$\mu = {\sum_{z \in \integers_+^{d_1}\times \integers^{d_2}}} \nu_z$$
where there is some probability measure $\nu$ supported on $M$ so that 
$\nu_z(y, w) = \one_{w = z} \nu(y)$ for all $z \notin \cB$, and there is a constant $A > 1$ so that
$\nu_z$ is a finite measure of mass 
\begin{equation}
\label{eq:nuzbound}
|\nu_z| < A,
\end{equation}
supported on $D_z$ for all
$z \in \cB$.
}
\smallskip


Let 
$$(y(x), z(x))=\begin{cases} ({\bf y}, {\bf z}) & \text{if }x=({\bf y}, {\bf z}),
{\bf y} \in M, {\bf z } \in
(\integers_+^{d_1}\times \integers^{d_2}){ \setminus \cB}\\
(\infty, \infty) & \text{if }{ z(x) \in \cB.} \end{cases} $$ 

{Here, $\infty$ is to be thought of as a label for the
bad part of the phase space.

\begin{defn}
  $\tT$ satisfies the almost mixing LLT (AMLLT)
  if there is
a bounded continuous function
$\fp:\reals_+^{d_1}\times \reals^{d_2} \to [0, \infty)$ satisfying \eqref{P-Prob} 
such that properties (a) and (b) below
hold. \smallskip

\noindent(a)  Let
$\brnu_{\phi, w}$ denote the measure 
defined by 
\begin{equation}
\label{DefNuZ}
d \brnu_{\phi, w}(y, z) = \phi(y) \one_{z = w} {{d\nu(y)}} ,
\end{equation}
where $w \in \integers_+^{d_1} \times \integers^{d_2}  \setminus \cB$
and $\phi: M \to  \reals$ is a Lipschitz function.
Then for every $\eps >0$ and  every $R
\in \reals$,
\begin{equation}
\label{AbsLLT}
 \lim_{n\to\infty}
 \sup_{\fA_{R, \eps}}
\left|  L_n^{d_1+d_2} 
\brnu_{\phi, w } 
\left( \psi(y({ \tT}^n x)) 
1_{z({ \tT}^n x)=  \lceil \bz L_n \rceil 
-  w
}\right) - 
\fp(\bz - w/L_n) \nu(\psi) \nu(\phi) \right| = 0 
\end{equation}
where 
the supremum in $\fA_{R, \eps}$ is taken over all 
quadruples $(\phi, \psi, w, \bz)$ where
$\phi$ and $\psi$ are Lipschitz functions on $M$ satisfying 
$$\|\phi\|_{Lip}\leq R, \quad
\|\psi\|_{Lip}\leq R,\quad 
 w \in \left(\integers_+^{d_1}\times \integers^{d_2}\right)  \setminus \cB, \quad
  \bz \in   [0, \infty)^{d_1}\times \reals^{d_2}, $$
  $$
  \left|\bz - \frac{w}{L_n}\right| < R,  \quad
 \dist(L_n \bz, \cB) > \eps L_n . $$
\smallskip

{
\noindent(b)  Let
$\brnu_{ \phi, w}$ denote the measure 
defined by 
\begin{equation}
\label{DefNuZB}
d \brnu_{\phi, w}(y,z) = \phi(y) \one_{z = w} d{\nu_w}
\end{equation}
where $w \in \cB$, $\phi: D_w \to  \reals$ is a Lipschitz function.
Then for  and every $w \in \cB$, every Lipschitz function $\phi: D_{ w} \to \reals$, every $\eps >0$, every $R
\in \reals$,
\begin{equation}
\label{AbsLLTB}
 \lim_{n\to\infty}
 \sup_{\fB_{R, \eps}}
\left|  L_n^{d_1+d_2} 
\brnu_{\phi, w } 
\left( \psi(y( \tT^n x)) 
1_{z( \tT^n x)= \lceil \bz L_n \rceil}\right) - 
\fp(\bz) \nu(\psi) \nu_{w}(\phi) \right| = 0 
\end{equation}
where 
the supremum in $\fB_{R, \eps}$ is taken over all 
pairs $(\psi, \bz)$ where
 $\psi$ is Lipschitz functions on $M$ satisfying 
$$\|\psi\|_{Lip}\leq R, \quad
   \bz \in   [0, \infty)^{d_1}\times \reals^{d_2}, \quad 
|\bz| < R,  \quad \dist(L_n \bz, \cB) > \eps L_n . $$}
\end{defn}

The AMLLT is the first version of our approximate periodic assumptions
and it deserves some commentary. 
The reader should think of "non-periodic part" $\cup_{z \in \cB} (D_z \times \{ z\})$ as being "negligible" 
compared to the "periodic part" 
$ \cup_{z \in \integers_+^{d_1}\times \integers^{d_2} \setminus \cB }(M \times \{ z\})$.

The condition \eqref{eq:cBcond} implies that most cubes of size 
$\xi Q$ in the cube of size $Q$ centered at origin are disjoint from
$\cB$. In fact, in all of our applications, either $\cB$ is a single point (local perturbations of a periodic system)
or $d_1 = d_2 = 1$ and 
$\cB = \{ 1 \} \times \integers$ (systems with boundary conditions). In these examples, it is immediate to check
\eqref{eq:cBcond}.
However, we present the general condition \eqref{eq:cBcond} because the proof of
the forthcoming Theorem \ref{ThALLT-Mix} is not easier in the special cases of $\cB$ as above and we
want to allow for more general framework to accomodate 
systems with boundary conditions and with sparse local impurities
which might be a subject of a future work.

Note that in \eqref{AbsLLT}, one observable is encoded in the
density of $\bar \nu$ (as compared with the formulation of the MLLT).
We also observe that while we require the convergence in
\eqref{AbsLLT} to be uniform
in the initial position $w$ and the initial density $\phi$, 
we do not  require this uniformity in \eqref{AbsLLTB}. 
Consequently \eqref{AbsLLTB} is simpler:
we may assume that $n$ is so large that $\| w \| \ll L_n$). 
This is because \eqref{AbsLLTB} is only used in the proof of local global mixing where the
initial density is fixed while \eqref{AbsLLT} is needed for global global mixing
and in the latter case one needs to decompose global observables as a sum of local ones,
which requires the uniformity of the convergence. See Section \ref{ScProofs} for more details.

 Using $\tX$ instead of $X$ and $\tilde \mu$ instead of $\mu$,
we can define $\bbG_U, \bbG_O, \bbG_{AO}$ as before with $d = d_1 + d_2$.
Namely, in the case $d_1 = 0$, the definition is the same with $d = d_2$.
If $d_1 >0$, we just need to accommodate for the fact that certain coordinates need to be positive. 
That is, in the definition of $\bbG_O$, $a_1, ..., a_{d_1}$ are assumed to be non-negative.
In the definition of $\bbG_U$, we consider cubes 
$V = M \times ( z + (-\lfloor w /2 \rfloor, w -\lfloor w /2 \rfloor]^d)$
where $z \in \integers^{d_1 + d_2}$, $z_1, ..., z_{d_1} > 0$ and $w \in \integers_+$ satisfies
$w < z_1,..., w< z_{d_1}$.
Finally, in the definition of $\bbG_{AO}$,
$\cG_{b,B}$ denotes the set of points
$z \in ((-b/2,B - b/2]^{d_1} \times (-B/2,B/2]^{d_2} \cup \integers^d)$ so that the cube centered at $z$
and of size $b$ satisfies \eqref{eq:globalmixdef}. 

The definition of global-global mixing
is the same as before, using the measure $\tilde \mu$. 
In the 
definition of local-global mixing \eqref{LocGlobMix},
we allow any function $\phi$ which is in $L^1(\tilde \mu)$

We think about $\cB$ as 
"small" , as exemplified by \eqref{eq:cBcond} and by the following observation.
The definitions
of $\bbG_O$ and  $\bbG_{AO}$ only depend on the "periodic part" of $\tX$ in the sense that
if we change a function $\Phi$ on the set $\cup_{z \in \cB} D_z \times \{ z\}$ (so as the new function is still
bounded and uniformly continuous), then it will not
affect whether $\Phi \in \bbG_{O}/ \bbG_{AO}$ holds or not. This follows from the condition \eqref{eq:cBcond}}.

We have the following result. 

{
\begin{theorem}
  \label{ThALLT-Mix}
 (a)  If $\tT$ satisfies the AMLLT,
then it enjoys local global mixing with respect to $\bbG_O.$

 (b)  If $\tT$ satisfies the AMLLT, 
then it enjoys global global mixing with respect to $\bbG_{AO}.$
\end{theorem}
}

\subsection{Approximately periodic systems}

Next { we study global mixing for
maps} which are asymptotically periodic at infinity.
{Thus we consider a periodic map $T$ on the set $X$ preserving the periodic measure $\mu$ as in 
\S \ref{ScResults1}.
In the setup of the next theorem,
global-global mixing of $\tT$ 
 is defined using the averages with respect to $\mu$,
 which need not be preserved by $\tT$.}

\begin{proposition}
\label{PrLocPert} 
If $T$ is a {periodic} map of a space $X$ preserving an infinite measure $\mu$ which is global global mixing {with respect to either $\bbG_{AO}$
or $\bbG_U$}
and
if $\tT$ equals to $T$ away from a finite {$\mu$-}measure set, then $\tT$ is also global global mixing
{with respect to the same space}.
\end{proposition}

In the remaining part of Section \ref{ScResults}, we discuss more drastic perturbations. The statements in this part of this section are unavoidably more technical. In fact, in our formulations we had two (somewhat conflicting) goals. First we wanted to facilitate the verifications of our abstract conditions for
  specific models of Section \ref{secEx}. Secondly, we wanted to emphasize that the proofs of our more technical results are very similar to the proofs
  for simpler periodic models. We advise the reader to {consult Sections \ref{sec:3.2} and \ref{secEx} for a complete
  understanding of the role of the} technical conditions imposed below.
  
\begin{defn}
\label{DefWApp}  
{ Let $T$ be
a periodic map on the set $X = M \times \integers^d$ preserving the periodic measure $\mu$ as in Section \ref{ScResults1}. Let $\tT$ be a map on
$X$.}
We say that $\tT$ is
  {\em very well approximated by $T$ at infinity} if $\tT$ preserves $\mu$ 
{and}

(i) For each $\eps>0$ there exists $R$ such that for each $|z|>R$ 
there is a set $A_{z, \eps}\subset M$ such that $\mu(A_{z, \eps})<\eps$ and for
{all} $y \not\in A_z$,
\begin{equation}
\label{eq:vwapproxcond}
d(\tT(y, z), T(y,z))<\eps.\end{equation}
\end{defn}

\begin{defn}
\label{DefWApp2} 
{Let $T$ be as in Definition \ref{DefWApp} and $\tT$ be a map
on
$\tX = D \cup ( M \times \integers_+^{d_1} \times \integers^{d_2})$,
where $D$ is a compact metric space.}

We say that $\tT$ is {\em well approximated by $T$ at infinity} if 
$\tT$ preserves
a measure $\tilde \mu$ such that  { $\tilde \mu(D) < \infty$} and for any $\eps>0$ there is $\delta=\delta(\eps)>0$ which
satisfies the following:
if $V$ is a box 
{ centered at 
$z \in \integers_+^{d_1} \times \integers^{d_2}$
of size $w \in \integers_+$ so that for all 
$i = 1,..., d_1 + d_2$, $w < |z_i| \delta$, then}
\begin{equation}
\label{GoodDensity}
 \sup_V \frac{d\tmu}{d\mu} \leq (1+\eps) \inf_V \frac{d\tmu}{d\mu}
\end{equation}
{ and moreover either (i) or (ii) holds, where 

(i) is as in Definition \ref{DefWApp} (in particular $d_1=0, d_2=d$) and

(ii) $d_1>0$, 
\eqref{eq:vwapproxcond} holds for $|z_1| > R$.
}

\end{defn}

{ Observe that if the measure $\tmu$ satisfies \eqref{GoodDensity} then the 
spaces of the global observables defined with respect to $\mu$ and $\tmu$ coincide
(and the infinite volume averages $\brPhi$ are the same). Therefore 
we will suppose in that follows that the spaces $\bbG_U$ and $\bbG_{AO}$ below are
defined using the invariant measure of the system as a reference measure.
}

\begin{theorem}
\label{ThInfAppGGM}
Suppose that $\tau$ is bounded and both $\tau$ and $T$ are almost everywhere continuous.

(a) 
If $\tT$ is very well approximated by $T$ at infinity and $T$ is global global mixing 
with respect to either {$\bbG_{AO}$} or $\bbG_U$, then 
$\tT$ is global global mixing with respect to the same space.

(b) 
If $\tT$ is well approximated by $T$ at infinity and $T$ is global global mixing 
with respect to $\bbG_U$, then so is
$\tT$.
\end{theorem}
{Note that in case of more general perturbations as in
Theorem \ref{ThInfAppGGM}, we can only guarantee global global mixing. See the beginning of \S \ref{SSSLocPert}
for a counterexample to local global mixing in the same setting.}

Next we provide { sufficient} conditions for local global mixing. 
{ Let $T$ be
a periodic map on the set $X = M \times \integers^d$ preserving the periodic measure $\mu$ as in Section \ref{ScResults1} and let $\tT$ be a map on
$X$} preserving a measure $\tmu$ satisfying \eqref{GoodDensity}.
{ The notion of global function is, as discussed above, the same whether using $\mu$ or $\tilde \mu$ in the 
definition. Now we study local-global mixing with respect to
$\tilde \mu$, that is, $\mu$ is replaced by $\tilde \mu$ in 
\eqref{LocGlobMix}.}
We assume that there is
a class $\fM$ of probability measures on $X$ 
and for each $\eps>0$ there is a class $\fM_\eps$ of probability measures on $M$ such that

(M1) {\bf (Invariance)} $\tT$ preserves $\fM.$

(M2) {\bf (Density)} For each compactly supported Lipschitz function $\phi$ and for each $\eps>0$ there is a finite set of functions
$\phi_1,\dots, \phi_k\in  L^\infty(X)\cap L^1({ \tmu})$ supported on
 the unit neighborhood of the support of
$\phi$ and constants $c_1,\dots c_k$ such that
$\displaystyle \left\Vert\phi-\left(\sum_{j=1}^k c_j \phi_j\right) \right\Vert_\infty\leq \eps$ 
and
$\phi_j { \tmu}\in \fM.$

(M3) {\bf (Approximation)} 
For each $\eps>0$ and $n\in \naturals$ there exists $R>0$ such that for each $\fm\in \fM$
$$ \fm(x: |z(x)|\geq R\text{ and } d(T^n x, \tT^n x)\geq \eps)\leq \eps. $$

(M4) {\bf (Uniform LLT)}
The measures from $\fM_\eps$ satisfy uniform LLT in the sense that for each
$\phi\in C(M)$, for each $K$ and for each $z_n$,
$$ L_n^d \fm\left(\phi(f^n x) \one_{z(T^n x)=z_n}\right)-\fp(z_n/L_n) \nu(\phi)\to 0 $$
and the convergence is uniform for $\fm\in \fM_\eps$ and $|z_n|/L_n\leq K$.

(M5) {\bf (Regularity Improvement)}
There is a constant $C<\infty$ such that for each $\fm\in\fM$ and each $\eps>0$ there exists
 $n_0=n_0(\fm, \eps)$ such that for all $n\geq n_0$ there is a decomposition
$ \tT_*^n \fm=\sum_j (c_j' \fm_j'+c_j'' \fm_j'') $ where $\fm_j', \fm_j''$ are supported on 
{ $M \times \{z=j\}$. Furthermore, 
for all $j$,  $\fm_j'$, when viewed as a measure on $M$
(with $ z = j$ fixed), is in the set}
$\fM_\eps$ and $\sum_j c_j''\leq C \eps.$

(M6) {\bf (Dissipation)}
For each $\fm\in \fM$ and for each $R>0$, 
$$\fm(|z(\tT^n x)|\leq R)\to 0\text{ as }n\to\infty.$$


We observe that while conditions (M1)--(M6) are logically independent of well approximation 
property
(Definition \ref{DefWApp2}),
condition (M3) has the same flavor as properties (i) and (ii) in that definition.

\begin{theorem}
\label{ThLGM-App} If $T$ and $\tT$ satisfy
(M1)-(M6), then $\tT$ is local global mixing with respect to $\bbG_U.$
\end{theorem} 

\section{Proofs}
\label{ScProofs}
Let $\bbL$ be the space of compactly supported Lipschitz functions on $X$. 
Note that $\bbL$ is dense in $L^1(\mu)$ so a standard approximation argument shows that
it suffices to prove \eqref{LocGlobMix} for $\phi\in \bbL.$ Henceforth we will suppose that all local functions
are in $\bbL.$

\subsection{ Periodic and almost periodic systems}
\label{SSCocycle}

\begin{proof}[ Proof of Theorem \ref{LLTIMMIx}(a)]
Let $\phi\in \bbL$, $\Phi\in \bbG_O. $ 
Since $\phi$ is compactly supported, we have 
$\phi(y,z) = \sum_k \phi_k(y) \one_{z = z_k}$ with a finite sum. 
Thus it suffices to prove the statement for 
{the function $(y,z) \mapsto \phi_k(y) \one_{z = k}$}, which for brevity is denoted by $\phi$ in the sequel. By the definition of
$\bbG_O$, { for every given $\eps>0$,
$R$ and $\delta$, there exists $K_0$ such that
the following property holds for all $K > K_0$:

{\bf (H)} for any cube $V$ of size $\delta K$
whose center is within $RK$ from the origin, we have 
\begin{equation}
\label{ArbCenter}    
  \left| \frac{1}{\mu(V)} \int_V \Phi d\mu - \brPhi \right| \leq \eps.
\end{equation}

Now choose $R$ such that 
\begin{equation}
\label{ChooseR}
\int_{|z|\geq R} \fp(z) dz<\eps.
\end{equation}
Then for large $n$, {the MLLT implies}
$$ \nu( y: |\tau_n(y)|\geq L_n R)<2\eps. $$
Thus
$$ \left| \int \phi(x) \Phi(T^n x) d\mu - \int \phi(x) \hPhi(T^n x) d\mu \right| \leq 2 ||\phi||_\infty ||\Phi||_\infty \eps, $$
where $\hPhi=\Phi \one_{|z|\leq R L_n}.$
Let $\Phi_{m }=\Phi \one_{z= m}$ for
$m \in \mathbb Z^d$. By the foregoing discussion,
\begin{equation}
\label{eq:proof1dec} \left| \int \phi(x) \Phi(T^n x) d\mu- \sum_{{|m|}\leq R L_n} \int \phi{(x)} \Phi_{m}(T^n x) d\mu \right|
\leq 2 ||\phi||_\infty ||\Phi||_\infty \eps .
\end{equation}
By the MLLT, 
{there exists a sequence of positive 
real numbers $\xi_n \to 0$ so that
for every $m \in \integers^d$ with $|m| < R L_n$,
$$
\left|
\int \phi(x)\Phi_m(T^nx)-L_n^{-d}\mu(\phi)\mu(\Phi_m) \fp(m/L_n)
\right|\leq
\xi_n L_n^{-d}.
$$
Summing this estimate for all $m$ as above and combining with 
\eqref{eq:proof1dec}, we obtain
\begin{align*}
& \left| \int \phi(x) \Phi(T^n x) d\mu- \sum_{{|m|}\leq R L_n} 
L_n^{-d} \mu(\phi) \mu(\Phi_m) \fp(m/L_n)
\right|  \\
& \leq 2 ||\phi||_\infty ||\Phi||_\infty \eps
+ R^d \xi_n.
\end{align*}
Hence in order to prove Theorem \ref{LLTIMMIx}(a), it suffices to verify that
\begin{equation}
\label{eq:thm2.3(a)est}
\limsup_{n \to \infty} \left|
L_n^{-d}
\sum_{{|m|}\leq R L_n} \mu(\Phi_m) \fp(m/L_n) 
- \bar \Phi \right| \leq \eps(2 R^d + 2 \bar \Phi).
\end{equation}
To this end,}
divide $\{ z \in \integers^d: |z|\leq R L_n\}$ into boxes $\cC_j$ of size $\delta L_n$.
Let $z_j$ be the center of $\cC_j.$
{First, since $\fp$ is uniformly continuous on the ball of radius
$R$, we can choose $\delta$ small so that for every $j$,
\begin{equation}
\label{eq:boxes1}
 \left| \sum_{m\in \cC_j} \mu(\Phi_m) \fp(m/L_n) - \sum_{m\in \cC_j} \mu(\Phi_m) \fp(z_j / L_n) \right| \leq \eps \mu(M \times \cC_j). 
\end{equation}
Next, by property {\bf (H)}, we have
\begin{equation}
\label{eq:boxes2}
 \left|
\left[\sum_{m\in \cC_j} \mu(\Phi_m) \fp(z_j / L_n) \right]- \fp(z_j / L_n) \brPhi \mu(M \times \cC_j)
\right| \leq \eps \mu(M \times \cC_j)
\end{equation}
Combining \eqref{eq:boxes1} and \eqref{eq:boxes2} and summing over $j$, we obtain
\begin{equation}
\label{eq:boxes3}
\left|
L_n^{-d}
\sum_{{|m|}\leq R L_n} \mu(\Phi_m) \fp(m/L_n)  
- \bar \Phi  \sum_j \fp(z_j / L_n) \delta^d \right| \leq
2 \eps R^d.
\end{equation}
Since \eqref{eq:boxes3} holds for an arbitrary small $\delta$ (provided that $n$ is large enough)
we can let $\delta\to 0$ thus
replacing the second sum by a Riemann integral.
Using 
\eqref{ChooseR}, we obtain
\eqref{eq:thm2.3(a)est} completing the proof of Theorem \ref{LLTIMMIx}(a)}}. 
\end{proof}

\begin{proof}[ Proof of Theorem \ref{LLTIMMIx}(b)]

{
In part (b) we prove a slightly stronger result, namely
we only assume that $\Phi_1 \in \bbG_O$. 
Let us fix
$\Phi_1 \in \bbG_O$, 
 $\Phi_2 \in \bbG_{AO}$ and $\eps >0$. We will show that there exists
$n_0$ and $B_0$ so that for all
$n > n_0$ and $B> B_0$, we have
\begin{equation}
\label{eq:thm3bend}
\left|
\frac{1}{\mu(V)} \int_V \Phi_1(x) \Phi_2(T^n x) d\mu - \bar \Phi_1 \bar \Phi_2 
\right| < \eps
\end{equation}
for any cube $V$ of size $B$ containing $M \times \{ 0 \}$. 
In fact, we will choose some auxiliary parameters $R = R(\Phi_1, \Phi_2, \eps)$ and
$\eps' = \eps'(\Phi_1, \Phi_2,R, \eps)$ before choosing $n_0
= n_0(\Phi_1, \Phi_2, \eps, R, \eps')$ and $B_0 = B_0(\Phi_1, \Phi_2, \eps, R, \eps')$.
To
simplify notation, let us write $z \in V'$ if $z \in \integers^d$
and $M \times \{ z\} \subset V.$
To prove \eqref{eq:thm3bend},
we use the decomposition
\begin{equation}
\label{GlToLoc}
\frac{1}{\mu(V)} \int_V \Phi_1(x) \Phi_2(T^n x) d\mu(x)
\end{equation}
$$=\frac{1}{\mu(V)}
\sum_{z\in V'} \sum_{w\in \integers^d} \int_M \Phi_1(y, z) \Phi_2(f^n y, w) 1_{\tau_n(y)=w-z} d\nu(y) . $$
We analyise the right hand side of \eqref{GlToLoc} in $6$ steps.\\

{\sc Step 1.} Take $R$ so large that for $n$ sufficiently large, the probability that $|\tau_n|>RL_n $ is smaller than 
$\DS \frac{\eps}{10 \|\Phi_1\|_\infty \|\Phi_2\|_\infty} $.
Such $R$ exists as in the proof of Theorem \ref{LLTIMMIx}(a). Then we can restrict the sum in 
\eqref{GlToLoc} to pairs such that $|w-z|\leq R L_n$ with an error which is at most $\frac{\eps}{10}.$ \\

{\sc Step 2.} Since $f$ satisfies the MLLT, we can replace the terms with $|w-z|\leq R L_n$
by
$$ \frac{1}{L_n^d}  \left(\int_M \Phi_1(y, z) d\nu(y) \right) \left( \int_M \Phi_2(y, w) d\nu(y) \right)
\fp\left(\frac{w-z}{L_n}\right)
$$
so that the total error we make in the sum \eqref{GlToLoc} does not exceed $\frac{\eps}{10}.$ 
Indeed, by the MLLT the error for any pair
$w,z$ with $|w-z|\leq R L_n$ is less than $\frac{\eps}{10R^dL_n^d}$ for
$n$ large. So far we derived
\begin{equation}
\bigg| 
\frac{1}{\mu(V)} \int_V \Phi_1(x) \Phi_2(T^n x) d\mu(x) - \label{eq:step2sum}
\end{equation}
$$\frac{1}{\mu(V)} \sum_{|w-z|\leq R L_n}
\frac{1}{L_n^d}  \left(\int_M \Phi_1(y, z) d\nu(y) \right) \left( \int_M \Phi_2(y, w) d\nu(y) \right)
\fp\left(\frac{w-z}{L_n}\right)\bigg| \nonumber
< \frac{2\eps}{10}. $$

{\sc Step 3.} Let $\tV$ be the cube with the same center as $V$ such that the size of $\tV$ equals to 
the size of $V$ plus $2L_n R.$ Denote
$$ \eps'=\frac{\eps}{10 \times  2^{d+1}R^d(\|\Phi_1\|_\infty+1) ( \|\Phi_2\|_\infty+1)(\|\fp\|_\infty+1)} . $$ 
Recall now the definition of $\bbG_{AO}$ with the corresponding 
functions $b(.), B_0(.)$ and set $\cG_{b,B}$. First, we let
$U$ be the cube centered at 0 and size $b(\eps'/2)$. Next,
assume that the size of $V$ is bigger than $B_0 := B_0(\eps'/2)$.
Given $\tz\in U$ let 
$$\cL_\tz=\{w\in V: w_i\equiv \tz_i \text{ (mod } b)\;\; \forall i=1,...,d\}.$$ Since the average proportion 
of $\cG :=\cG_{b(\eps'/2),  \text{ size }(V)}(\Phi_2)$ in $\bigcup_{\tz\in U} \cL_\tz$ is greater than $1-\eps'/2$ there exists $\brz$ such that  
the proportion of $\cG$ in $\cL_\brz$ is greater than $1-\eps'/2.$ 
Let $\{U_j\}$ be the collection of cubes of size $b$ 
whose centers are congruent to $\brz$ mod $b$ 
and which intersect $\tV.$ Note that $U_j$'s are disjoint and 
their union contains $\tV$.
Let $\fG$ be the union of $U_j$ which are completely contained in $V$
such that 
\begin{equation}
\label{Phi2GAO}
\left|\frac{1}{\mu(U_j)} \int_{U_j} \Phi_2(x) d\mu(x)-\brPhi_2\right|\leq \frac{\eps'}{2} 
\end{equation}
and $\fB$ be the complement of $\fG$ in $\tV$
($\fG$ and $\fB$ stand for "good" and "bad"). Since the size of $V$ is larger than $B_0$, we have
\begin{equation}
\label{FBLowDensity}
  \mu(\fB)\leq \eps' \mu(V) 
\end{equation} 
(we replaced $\eps'/2$ by $\eps'$ in the RHS to account for boundary effects,
that is, the cubes which are not completely
contained in $V$).
\\

{\sc Step 4.} If $n$ is sufficiently large, then the oscillation of $\fp$ on the boxes of size
$b(\eps'/2)/L_n$ is smaller than $\eps'$. Let us denote by $u_j$ 
the centers of $U_j$. Then by the definition of $\eps'$, 
we can replace \eqref{eq:step2sum} by 
$$  \frac{1}{\mu(V) L_n^d}  \sum_{z\in V'} \int_M
\Phi_1(y, z) d\nu(y) 
\sum_{j: d(z, u_j)\leq R L_n} \fp\left(\frac{u_j-z}{L_n}\right)
\sum_{w\in U_j} \int \Phi_2(y, w) d\nu(y)=
$$
\begin{equation}
\label{eq:step4sum}
 \frac{1}{\mu(V) L_n^d}  \sum_{z\in V'} \int_M
\Phi_1(y, z) d\nu(y) 
\sum_{j: d(z, u_j)\leq R L_n} \fp\left(\frac{u_j-z}{L_n}\right)
\int_{U_j} \Phi_2(x) d\mu(x)
\end{equation}
with an error smaller than $\frac{\eps}{10}.$
\\

{\sc Step 5.} Next, we estimate the error made when replacing $\int_{U_j} \Phi_2(x) d\mu(x)$ in \eqref{eq:step4sum} by $\mu(U_j) \brPhi_2$
for all $z$ and $j$.  First, the error
introduced by all $j,z$ so that $U_j \subset \fG$ is at most 
$$
\frac{1}{\mu(V) L_n^d} \sum_{z\in V'} \| \Phi_1\|_{\infty}
\sum_{j: d(z, u_j)\leq R L_n, U_j \subset \fG} \| \fp \|_{\infty} \frac{\eps'}{2}
\leq \frac{\eps}{10},
$$
where we used \eqref{Phi2GAO} and the definition of $\eps'$.
Secondly,  the error
introduced by all $j,z$ so that $U_j \subset \fB$ is at most 
\begin{align*}
& \frac{1}{\mu(V) L_n^d} \sum_{z\in V'} \| \Phi_1\|_{\infty}
\sum_{j: d(z, u_j)\leq R L_n, U_j \subset \fB} \| \fp \|_{\infty} 2 \mu(U_j) \| \Phi_2
\|_{\infty} \\
& \leq \frac{ 2 \| \Phi_1\|_{\infty} \| \fp \|_{\infty}  \| \Phi_2\|_{\infty}}
{\mu(V) L_n^d} \sum_{j:U_j \subset \fB } \mu(U_j) \sum_{z \in V':  d(z, u_j)\leq R L_n} 1 \\
& \leq  \frac{ 2^{d+1} R^d \| \Phi_1\|_{\infty} \| \fp \|_{\infty}  \| \Phi_2\|_{\infty}}
{\mu(V)} \mu(\fB)
\leq \frac{\eps}{10},
\end{align*}
where the penultimate inequality uses that there are at most $(2 R L_n)^d$ points
$z$ with $d(z, u_j)\leq R L_n$ and the last inequality follows from \eqref{FBLowDensity} and the definition of $\eps'$. Recalling steps 2 and 4, we arrive at
\begin{equation}
\label{eq:step5sum}
\bigg| \frac{1}{\mu(V)} \int_V \Phi_1(x) \Phi_2(T^n x) d\mu(x) - 
\end{equation}
$$
 \frac{1}{\mu(V) L_n^d}  \sum_{z\in V'} \int_M
\Phi_1(y, z) d\nu(y) 
\sum_{j: d(z, u_j)\leq R L_n} \fp\left(\frac{u_j-z}{L_n}\right)
\mu(U_j) \bar \Phi_2\bigg| 
< \frac{5\eps}{10}.
$$

{\sc Step 6.} Noting that $\mu(U_j)=b^d$, it remains to evaluate
$$ \frac{b^d}{\mu(V) L_n^d}  \sum_{z\in V'} \int_M
\Phi_1(y, z) d\nu(y) 
\sum_{j: d(z, u_j)\leq R L_n} \fp\left(\frac{u_j-z}{L_n}\right)
\brPhi_2.$$

For large $n$, the Riemann sum 
$\DS \frac{b^d}{L_n^d} \sum_{j: d(z, u_j)\leq R L_n} \fp\left(\frac{u_j-z}{L_n}\right)$
can be replaced by the integral 
$\DS \int_{|t|<R} \fp(t) dt $ with an error smaller than 
$$ \frac{\eps}{10\| \Phi_1 \|_{\infty}\| \Phi_2 \|_{\infty}}.$$ 
The last integral is in the interval $(1-\frac{\eps}{10\| \Phi_1 \|_{\infty}\| \Phi_2 \|_{\infty}}, 1]$ by our choice of $R$. Thus we arrive at 
$$ \left| \frac{1}{\mu(V)} \int_V \Phi_1(x) \Phi_2(T^n x) d\mu(x)- \frac{1}{\mu(V)}
\sum_{z\in V'} \int \Phi_1(y,z) d\nu(y) \bar \Phi_2 \right|\leq \frac{7\eps}{10}. $$
 Finally, since $\Phi_1\in \bbG_O$, we have 
$$ \left| \frac{1}{\mu(V)}
\int_V \Phi_1(x) d\mu(x) -\bar \Phi_1 \right| < \frac{\eps}{10 \|\Phi_2\|_{\infty}}.$$ The last two displays imply \eqref{eq:thm3bend}. Theorem \ref{LLTIMMIx} (b) follows.}
\end{proof}

The proof of Theorem \ref{SLLTIMMIx} is  similar 
to the proof of Theorem \ref{LLTIMMIx} (a)
except that we need to consider boxes around $D_n$ rather than around the origin. {In fact, the proof of Theorem \ref{SLLTIMMIx} (b) is 
simpler than the
proof of Theorem \ref{LLTIMMIx} (b) because all points $w$ are good and we don't need the set $\fB$.}

{
\begin{proof}[Proof of Theorem \ref{ThALLT-Mix}]
The proof of Theorem \ref{ThALLT-Mix} is similar to that of
Theorem \ref{LLTIMMIx}. 
Recall that in the proof of Theorem \ref{LLTIMMIx} (a), we used the MLLT for 
$m \in \cC_j$,
where $\cC_j$ is a box of size $\delta L_n$ within distance $RL_n$ from the origin.
We could treat the contribution of $m$ with $|m| \geq L_n R$ as an error term by \eqref{ChooseR}.

 We start the proof of Theorem \ref{ThALLT-Mix} (a)
by assuming without loss of generality that $\phi$ is 
supported on $D_k \times \{k \}$ for some
$k \in \integers_+^{d_1} \times \integers^{d_2}$ as in the
beginning of the
proof of Theorem \ref{LLTIMMIx} (a). Note that now we will have 
to study both cases of $k \in \cB$ and $k \notin \cB$. We again 
choose $R$ as in \eqref{ChooseR} 
except that we replace $\eps$ by
$$
\eps' = \frac{\eps}{3 A (1+\|\Phi\|_\infty)(1+ \|\phi\|_\infty) },
$$
where $A$ is defined by \eqref{eq:nuzbound}.
Then the contribution of points $m$ with $|m| >R L_n $ is negligible.
Then we again partition the set
$m
\in \integers_+^{d_1} \times \integers^{d_2}$ 
with $|m|\leq RL_n $, into boxes $\cC_j$ of size $\delta L_n$. 
Let us write $j \in \cJ_1$ if $\dist(\cC_j, \cB) \geq \delta L_n$ and $j \in \cJ_2$
otherwise. Let us also write $d = d_1 + d_2$.

First we prove that the contribution of boxes $\cC_j$, $j \in \cJ_2$ is negligible. To this end, apply \eqref{eq:cBcond} with 
$$\eta = \frac{\eps'}{R^d\|\fp\|_{\infty}}.
$$
This gives us $\xi$ and $Q_0$. Now we choose $\delta < \xi R / (d+1)$ and $n$ big so that
$R L_n > Q_0$. Then by \eqref{eq:cBcond},
\begin{align}
 \sum_{j \in \cJ_2}
| \cC_j | 
&\leq | \{ \bk \in \integers_+^{d_1} \times \integers^{d_2},
|\bk | \leq RL_n: \dist(\bk, \cB) < (d+1)\delta L_n \} | \nonumber\\
& \leq \eta (RL_n)^d.\label{eq:thm2.8apriori}
\end{align}

Let $\DS \cB^*=\bigcup_{j\in \cJ_2} \cC_j$ be $\delta L_n$ neighborhood of $\cB$ in the 
box of size
$R L_n$ around the origin and $\DS \cG^*=\bigcup_{j\in \cJ_1} \cC_j$.
We have
$$ \left|  
 \sum_{j\in \cJ_2}
\int_{D_k \times \{k \}} \phi(x) \Phi_1(T^n x) 
\one_{z(T^n(x)) \in \cC_j}d\mu(x) \right|\leq 
 \|\phi\|_\infty \|\Phi\|_\infty  \nu_k(T^n x\in \cB^*) $$
\begin{equation}
\label{eq:Thm28asum2}
\leq \|\phi\|_\infty \|\Phi\|_\infty  
[ \nu_k(D_k) - \nu_k(T^n x\in \cG^*) ] 
\end{equation}
Applying the AMLLT (specifically, using
\eqref{AbsLLT} with $\phi = \psi = 1$ in case $k \notin \cB$
 and \eqref{AbsLLTB} with $\phi = 1/\nu_k(D_k)$, $\psi = 1$ in case $k \in \cB$),
we obtain that for large $n$ large
$$ 
\frac{\nu_k(T^n x\in \cG^*)}{\nu_k(D_k)} =L_n^{-d} \sum_{ j\in \cJ_2} (\delta L_n)^d \left[\fp(z_j/L_n)+\kappa_{j,n}\right]=
 \sum_{ j\in \cJ_2} \delta ^d \left[\fp(z_j/L_n)+\kappa_{j,n}\right]
$$
where $z_j$ are the centers of $\cC_j$ and the error term 
$\DS \sum_{j\in \cJ_2} \kappa_{j, n}$ can be made as small as we 
wish by taking $n$ large. Making $\delta$ small we can make the last sum arbitrarily close to
$$ \int_{\cG^*/L_n} \fp(z) dz=1-\int_{|z|\geq R} \fp(z) dz-\int_{\cB^* / L_n}\fp(z) dz. $$
Both integrals on the right hand side of the last display are smaller than $\DS \frac{\eps}{3
A\|\phi\|_\infty\| \Phi\|_\infty}$:
the first one
due to our choice of $R$, and the second one 
due to our choice of $\eps', \eta$ and \eqref{eq:thm2.8apriori}. 
Now combining the last two displays, we obtain
$$  \nu_k(D_k)
-\nu_k(T^n x\in \cG^*) \leq \frac{3 \eps}{3 \|\phi\|_\infty \|\Phi\|_\infty} $$
 which combined with \eqref{eq:Thm28asum2} shows
that the contribution of $\cJ_2$ is indeed negligible. 

The computation of the main term, namely the contribution of boxes $\cC_j$, $j \in \cJ_1$ is done along the lines of the proof of Theorem \ref{LLTIMMIx} (a).
Indeed, the AMLLT is applicable on those boxes. Theorem \ref{ThALLT-Mix} (a) follows.\\

The proof of Theorem \ref{ThALLT-Mix} (b)
is again similar to the proof of Theorem 
\ref{LLTIMMIx} (b)
so we only explain the differences and use the same notations as there.
In fact, in this proof we only use \eqref{AbsLLT} and won't need \eqref{AbsLLTB}.

We still prove \eqref{eq:thm3bend}, but now we allow $B_0$ to depend on $n$,
which is allowed by Definition \ref{def:gg}. 
Now
\eqref{GlToLoc} reads 
\begin{align}
&\frac{1}{\mu(V)} \int_V \Phi_1(x) \Phi_2(T^n x) d\mu(x) \nonumber\\
=&\frac{1}{\mu(V)}
\sum_{z\in V'} \sum_{w\in \integers^d} \int_{D_z} \Phi_1(y, z) \Phi_2(f^n y, w) 1_{\tau_n(y)=w-z} d\nu_z(y). \label{GlToLoc2}
\end{align}

First we show that the sum over $z$ that are close to the set $\cB$
is negligible.
To this end, we first apply \eqref{eq:cBcond} with 
$$\eta = \frac{\eps}{20 A2^d\| \Phi_1\|_{\infty}  \| \Phi_2\|_{\infty}}.$$
This gives us $\xi$ and $Q_0$. Now for an $n$, we will choose
$B_0$ so large that $B_0 > Q_0$ and $B_0 \xi > 2RL_n + b$.

Now let $V$ be a cube of size $B \geq B_0$ containing $M \times \{ 0 \}$.  Then $V$ is contained
in another box $\hat V$ of size at most $2B$ centered at the origin.
The contribution of $z \in V'$ with $\dist(z, \cB) < \xi B$ 
to the sum in \eqref{GlToLoc2}
is now bounded by
\begin{align*}
&\frac{1}{\mu(V)} \sum_{z \in V', \dist(z, \cB) < \xi B} \nu_z(D_z) \| \Phi_1\|_{\infty}  \| \Phi_2\|_{\infty}\\
&
\leq \frac{ A}{\mu(V)} \| \Phi_1\|_{\infty}  \| \Phi_2\|_{\infty} 
| \{  z \in \integers_+^{d_1} \times \integers^{d_2}: D_z \times \{ z\} \subset \hat V, 
\dist(z, \cB) < \xi B\}| \\
&\leq 
 \frac{ A}{\mu(V)} \| \Phi_1\|_{\infty}  \| \Phi_2\|_{\infty} \eta (2B)^d 
\leq \frac{\eps}{10},
\end{align*}
where the first inequality in the last line
follows from \eqref{eq:cBcond} applied to the box $\hat V$
and the last inequality follows
from the estimate
 $\mu(V) \geq \frac{B^d}{2}$ and the definition of $\eta$.

Thus the sum for $z$ with $\dist(z, \cB) < \xi B$ is negligible and
instead of \eqref{GlToLoc2} it is sufficient to study 
$$
\frac{1}{\mu(V)}
\sum_{z\in V', \dist(z, \cB) \geq \xi B} \sum_{w\in \integers^d} \int_{M} \Phi_1(y, z) \Phi_2(f^n y, w) 1_{\tau_n(y)=w-z} d\nu(y)
$$
(note that if $\dist(z, \cB) \geq \xi B$, then in particular $D_z = M$,
$\nu_z = \nu$).

Now we repeat Steps 1--6 of the proof of 
Theorem 
\ref{LLTIMMIx} (b) with two minor changes. First, in Step 1, 
we use the AMLLT instead of the MLLT.
Indeed, the AMLLT
is applicable because if $|w-z| < RL_n$, then
recalling the inequality $B \xi > 2RL_n + b$,
we also have $\dist(w, \cB) \geq RL_n$. Second,
in all of Steps 1--6, each
sum over $z$ is replaced by sum over $z$ with 
$\dist(z, \cB) \geq \xi B$. 
Since the sum over $z$ with $\dist(z, \cB) \geq \xi B$ is negligible as shown above,
this change introduces negligible additional errors to the estimates of Steps 1--6.
This completes the proof of Theorem 
\ref{ThALLT-Mix} (b).
\end{proof}

}

\subsection{Global global mixing for approximations.}

\begin{proof}[Proof of Proposition \ref{PrLocPert}:]
  Let $A=\{x: Tx\neq \tT x\}.$ Then
  \begin{equation}
    \left| \int_V \Phi_1(x) [\Phi_2(T^n x)-\Phi_2(\tT^n x)] d\mu \right| \label{eqlocpert1}
  \end{equation}  
$$   \leq 2 ||\Phi_1||_\infty ||\Phi_2||_\infty
  {\mu}(x: \exists 0\leq k<n : T^k x\neq \tT^k x)
  \leq 2 ||\Phi_1||_\infty ||\Phi_2||_\infty n \mu(A). $$
 
Since the last expression does not grow as $\mu(V)\to \infty$ we obtain the result.
\end{proof}  


\begin{proof}[Proof of Theorem \ref{ThInfAppGGM}]
(a) We will show that for each $n$
\begin{equation}
\label{LargeVolComp}
 \lim_{\mu(V)\to\infty} 
 \frac{1}{\mu(V)} 
\left[\int_V \Phi_1(x) \Phi_2(\tT^n x) d\mu-\int_V \Phi_1(x) \Phi_2(T^n x) d\mu\right]=0. 
\end{equation}
Note that for each $n,$ $T^n$ is continuous almost everywhere. 
Fix an arbitrary $n \in \mathbb N$ and $\varepsilon >0$.
An induction on $n$ shows 
that for $\nu$ a.e. $y$
there exists $\delta=\delta(y, \eps)$ such that
if 
$\{y'_k\}_{k=0}^n$ is a sequence such that 
$d(y'_0, y)<\delta$ and $d(f (y'_k), y'_{k+1})\leq \delta$, then
$$d(f^n (y), y'_n)\leq \eps\quad \text{and}\quad \tau_n(y)=\sum_{k=0}^{n-1} \tau(y'_k). $$
We will say that $y$ is $(\delta, \eps)$-good. Let $B_{n, \delta, \eps}$
be  the set of not $(\delta, \eps)$-good points.
Choose $\delta = \delta(\eps)$ so small that the measure of $B_{n, \delta, \eps}$
 is less than $\eps$
{(such $\delta$ exists by the continuity of the measure
as $\nu(\cap_{\delta >0} B_{n, \delta, \eps} ) = 0$).} 
Next, choose $R = R(\eps)$ such that for $|z|> R$ we have
$ \mu(A_{z, \delta})\leq \eps.$ 

We are now ready to establish \eqref{LargeVolComp}. 
To fix ideas let us suppose that $V$ is a cube of size $L.$
We split $V$ into two parts. Let
$V_1$ be the set of points 
$x=(y,z) \in V$ for which 
\begin{itemize}
\item
there is some $k \leq n$ so that the absolute value of the $z$-coordinate of $\tilde T^k x$ is less than $R$, or
\item there is some $k \leq n$ so that $\tilde T^k x \in \cup_{z} A_{z, \delta}$, or
\item $y \in B_{n, \delta, \eps}$.
\end{itemize}
Denote $V_2=V-V_1.$ Assume
$|\tau|\leq r.$ Then the orbit of points from $V$ are within distance $nr$ from $V.$
It follows that 
$$\mu(V_1)\leq (R+r)^d+2 (L+n r)^d n \eps+ \eps,$$
{\hskip-1.5mm
where the three summands above corresponds to the three cases in the definition of
$V_1$ above.}
Thus the contribution of $V_1$ to \eqref{LargeVolComp} is less than
$$ \left[(R+n r)^d+2 (L+n r)^d n \eps{+\eps}\right] ||\Phi_1||_\infty ||\Phi_2||_\infty.  $$
On the other hand if $(x,z)\in V_2$ then 
$d(T^n (x,z), \tT^n (x,z))\leq \eps$ and so 
the contribution of $V_2$ is less
$ \mu(V) ||\Phi_1||_\infty \Osc(\Phi_2, \eps)$ 
where 
$$ \Osc(\Phi, \eps)=\sup_{d(x', x'')\leq \eps} |\Phi(x')-\Phi(x'')|. $$
It follows that for large $L$
$$ \frac{1}{\mu(V)} 
\left|\int_V \Phi_1(x) \left[\Phi_2(\tT^n x) -\Phi_2(T^n x) \right] d\mu\right|$$
$$\leq
 3 n \eps ||\Phi_1||_\infty ||\Phi_2||_\infty+||\Phi_1||_\infty \Osc(\Phi_2, \eps). $$
Since $\eps$ is arbitrary, we can take the limit $\eps\to 0$ obtaining \eqref{LargeVolComp}.
This completes the proof of part (a). 

To prove part (b) we may assume that $V$ is such that 
$ \sup_V z \leq (1+\delta(\eps))\inf_V z.$ If this does not hold, we subdivide $V$ into smaller boxes
and remove the central part (which has small relative measure). Next we use \eqref{GoodDensity}
to replace 
$$  \frac{1}{\tmu(V)} 
\left[\int_V \Phi_1(x) \Phi_2(\tT^n x) d\tmu\right]
\text{ by }  
\frac{1}{\mu(V)} 
\left[\int_V \Phi_1(x) \Phi_2(\tT^n x) d\mu\right] $$
and then conclude as before using \eqref{LargeVolComp}.
\end{proof}

\subsection{Local global mixing for approximations.}
\label{sec:3.2}
\begin{proof}[Proof of Theorem \ref{ThLGM-App}]
  Due to (M2), it suffices to show that for each $\fm\in \fM$ and for each $\Phi\in \bbG_U$,
  we have 
$\fm(\Phi(\tT^n x))\to\brPhi$ as $n \to \infty$.

{
Let us fix some $\fm \in \fM$, $\Phi\in \bbG_U$ and 
$\eps >0$. We will show that for $n$ large enough,
\begin{equation}
\label{eq:Thm2.13pf}
|\fm(\Phi(\tT^n x)) - \brPhi|\leq (4 + C + 
\| \Phi\|_{\infty})\eps,
\end{equation}
where $C$ is the constant in (M5). To do so, we will choose a small parameter
} $\delta = \delta(\eps) >0$ and
large numbers $\brn = \brn (\eps)$, $R = R( \delta, \brn)$, 
$n = n(\eps, \delta, \brn,R) {\gg \bar n}$.
{We will apply  $\tT$ for $n - \bar n$ iterations. Then we will show that
during the remaining time $\bar n$, we can well approximate $\tT$ by $T$.

First, we prove the following preliminary estimate: for the already fixed $\eps >0$ there is
$\bar n$ so that for all $\fm' \in \fM_\eps$ and all $z\in \integers^d$
\begin{equation}
\label{eq:prelimest}
\left| \int \Phi(f^\brn y, z+\tau_\brn(y)) d\fm'(y)-\brPhi\right|\leq \eps. 
\end{equation}}
Indeed, \eqref{eq:prelimest} follows from 
(M4) and precompactness of the set $\{\Phi_l\}$ (where 
$\Phi_l(x) = \Phi(x,l)$), as in to the proof of Theorem \ref{LLTIMMIx}(a).

{Next, by equicontinuity of $\{\Phi_l\}$, there exists $\delta = \delta(\eps) \leq \eps$}
such that if 
$d(x', x'')<\delta$, then $|\Phi(x')-\Phi(x'')|<\eps.$

Let { us now define $\tfm=\tT_*^{n-\brn} \fm$}.
We claim that if $n$ is large enough, then
\begin{equation}
\label{LastLeg}
|\tfm(\Phi(T^\brn x))-\fm (\Phi(\tT^n x))| 
={
|\tfm(\Phi(T^\brn x))-\tfm(\Phi(\tT^{\brn} x))| }
\leq { 3\eps}. 
\end{equation}
{
The equation in \eqref{LastLeg} follows from the definition of $\tfm$. To prove the inequality,
let us write
\begin{eqnarray}
&&| \tfm(\Phi(T^\brn x) - \Phi(\tT^{\brn} x)) | \nonumber \\
&\leq &\tfm\left[
\one_{|z(x) > R|}
\left| \Phi(T^\brn x) - \Phi(\tT^{\brn} x)\right| \right] \label{eq:smallR} \\
&+&
\tfm\left[
\one_{|z(x) \leq R|}
\left|\Phi(T^\brn x) - \Phi(\tT^{\brn} x)\right| \right].\label{eq:bigR} 
\end{eqnarray}
Here, $R = R(\delta, \brn)$ is chosen so that}
$$ \tfm(x: |z(x)|>R\text{ and } d(\tT^\brn x, T^\brn x)> \delta)<\delta $$
(such $R$ exists by (M3)). 

{By the choice of $\delta$ and $R$, \eqref{eq:smallR} is bounded above by}
$$ 2 ||\Phi||_\infty \tfm(x: |z(x)|>R, d(\tT^\brn x, T^\brn x) > \delta)+\eps\leq 2 ||\Phi||_\infty \delta+\eps 
{\leq 2 \eps}
$$
{(note that we can assume without loss of generality that $\delta < \eps / ( 2 ||\Phi||_\infty)$).
Next, (M6) implies that \eqref{eq:bigR} is smaller than $\eps$ if $n$ is large enough.
We have verified \eqref{LastLeg}.

By \eqref{LastLeg}, it remains to estimate
$ \tfm(\Phi(T^\brn x))$. 
Assuming that 
$n - \bar n>n_0(\fm, \eps)$,
 where $n_0$ is defined in property (M5), we have}
$$\begin{aligned} \tfm(\Phi(T^\brn x))&=\sum_j (c_j' \fm_j'(\Phi(T^\brn x))+c_j'' \fm_j''(\Phi(T^\brn x)))\\
&=
\sum_j c_j' \fm_j'(\Phi(T^\brn x)) { + \cE}, \end{aligned}$$
where $\cE$ is an error term satisfying $|\cE| <
{ C} \eps$.
{ By  (M5) and \eqref{eq:prelimest}}, for each $j$
$$ |\fm_j'(\Phi(T^\brn x))-\brPhi|\leq \eps. $$
{ Next, by (M5),} 
$$1\geq \sum_j c_j'=1-\sum_j c_j''\geq 1-\eps. $$
{ Combining the last three displays, we derive
$$
|\tfm(\Phi(T^\brn x)) - \bar\Phi| \leq (1 + \| \Phi \|_{\infty} + C) \eps,
$$
which togather with \eqref{LastLeg} implies \eqref{eq:Thm2.13pf}.
The theorem follows.}
\end{proof}

  \section{Mixing for flows.}
\label{ScFlows}

The results of Section \ref{ScResults} can be extended to flows. Here, we briefly summarize the
necessary changes in the definitions and theorems.

Let $X=M\times\integers^d$,
$x = (y,z) \in X$ and $G^t(y,z)=(g^t(y), z+\tau^t (y))$ for $t\geq 0$ (or for $t \in \reals$) where $X$ is as before, and
$g^t$ preserves a probability measure $\boldsymbol\kappa.$
We equip $X$ with the measure $\lambda$ which is the product of $\boldsymbol \kappa$ and the counting measure
on $\integers^d.$ We define the spaces $\bbL, \bbG_O, { \bbG_{AO}}, \bbG_U$ as before. 

The definition of local-global and global-global mixing is
analogous, we just need to replace $T^n$ by $G^t$ and let $t \rightarrow \infty$ instead of $n \rightarrow \infty$.
Noting that the second coordinate of $X$ is still discrete, we can extend the definition of MLLT and shifted MLLT
by simply replacing 
$f^n$, $\tau_n$, $z_n^0 \in \integers^d$, $L_n$, $D_n$ and $n \rightarrow \infty$
by $g^t$, $\tau_t$, $z_t^0 \in \integers^d$, $L_t$, $D_t$ and $t \rightarrow \infty$ respectively. Similarly, we define 
AMLLT by replacing $\tT^n$, $z_n^0$, $L_N$ and $\lim_n$ by $ \tG^t$, $z_t^0$, $L_t$ and $\lim_t$ respectively. With these
adjustments,  
one can extend Theorems \ref{LLTIMMIx}--\ref{ThALLT-Mix} as well as their proofs to the case of
flows.

In the remaining results, the map $\tilde T$ was approximated by a periodic map $T$. In case of flows, we can define
similar approximations by, say, comparing the two flows up to time $1$. First, the following analogue of 
Proposition \ref{PrLocPert} holds:

\begin{proposition}
\label{PrLocPertFlow}
If $G^t$ is a flow on a space $X$ preserving an infinite measure $\boldsymbol{\kappa}$ which is 
global global mixing { with respect to either $\bbG_{AO}$
or $\bbG_U$}
and
if $\tG^t{ (x)}$ equals to $G^t {(x)}$ 
for { all} $t \in [0,1]$ 
{ and all $x$} away from a finite measure set, then $\tG$ is global global mixing.
\end{proposition}

We can obtain a proof of Proposition \ref{PrLocPertFlow} from the proof of Proposition \ref{PrLocPert} by 
replacing $A = \{ x: Tx \neq \tT x \}$ by $A = \{ x: \exists t \in [0,1]: G^t (x) \neq \tG^t (x) \}$,
and $n$ by $t$ in 
\eqref{eqlocpert1}. 

Similarly, in the definition of good and very good approximation, besides the obvious changes, we require that for all
$y \notin A_z$ and for all $t \in [0,1]$, $d( \tG^t(y,z), G^t(y,z)) < \eps$. Then we have

\begin{theorem}
\label{ThInfAppGGMFlow}
Suppose that $\{ \tau_t(y) : y \in M, t \in [0,1] \}$ is bounded and the set
$$\{y \in M: g_t(y) \text{ and } \tau_t(y) \text{ are continuous at $y$} \}$$
has full measure for any fixed $t$. 

(a) 
If $\tG$ is very well approximated by $G$ at infinity and $G$ is global global mixing 
with respect to either {$\bbG_{AO}$} or $\bbG_U$, then 
$\tG$ is global global mixing with respect to the same space.

(b) 
If $\tG$ is well approximated by $G$ at infinity and $G$ is global global mixing 
with respect to $\bbG_U$, then so is
$\tG$.
\end{theorem}

The proof of Theorem \ref{ThInfAppGGMFlow}
is similar to that of Theorem \ref{ThInfAppGGM} with minor changes as before. 
We leave the details to the reader.

Finally, the assumptions (M1)--(M6) can analogously be formulated for flows. Namely, (M1) claims that
 $\tG^t$ preserves $\fM$ for every $t$, (M2) is unchanged and all changes in (M3)--(M6) 
 amount to replacing $T, \tT$ by $G, \tG$ are as before. With these
 changes, and with a similar proof, we can derive the analogue of Theorem 
\ref{ThLGM-App}.

\section{Preliminaries on Lorentz gas and related systems.}
\label{BilPrel}

In the remaining part of the paper, 
we give several examples of systems satisfying the assumptions of Section \ref{ScResults}.
In those examples we have a {point mass} 
moving in $\reals^d$ with a number of scatterers 
removed and having elastic reflections from the boundary. The motion between the collisions
will be either free (such as in case of Lorentz gas) 
or subject to a field. In this case the most interesting question from physical
point of view is to study mixing properties of the continuous time system, however, mathematically
one could also study the mixing properties of the collision map, too. 
We will also use natural examples 
below to illustrate several subtleties associated to the notions of local global and global global mixing.

In our examples, the system having approximate symmetry will be denoted by $\tT$ while
its symmetric approximation will be denoted by $T.$ In the continuous time setting, the corresponding systems will be denoted by {$\tG^t$ and $G^t$,} respectively.

For the reader's convenience, we summarize some basic facts about Lorentz gas in this section.
We will focus on the notions and results that are most important for studying global mixing properties.
Everything in this section (as well as many other important results) can be found in
\cite{CM06}.
Thus we do not give more references. Much of the theory presented in this section
has been extended to billiards subject to external fields (see \cite{Ch01, C08, DZ13}).
Additional references
will be given later when we discuss specific examples.

Let $O_1,\dots, O_J$ be disjoint convex subsets of the 2-torus $\mathbb T^2$ with $\mathcal C^3$ 
boundary with non-vanishing curvature. These sets are also called scatterers.
Let us consider a point particle that flies freely (with speed $1$)
in the interior of $\cD_0 = \mathbb T^2 \setminus \cup O_j$, 
and, upon reaching the boundary, undergoes
specular reflection (angle of incidence equals angle of reflection). This dynamics is called the Sinai billiard flow
($g^t$). It preserves the Lebesgue measure on $\cD_0 \times \cS^1$ (position and velocity). Let
$\boldsymbol\kappa$
be the invariant Lebesgue measure normalized so as it is a probability measure.
Identifying the torus with $[0,1]^2$, and extending the scatterer configuration periodically to the plane,
we define the billiard flow on
$\cD = \reals^2 \setminus \cup_{\ell \in \integers ^2} \cup_{j = 1}^J (O_j+ \ell)$ as before. We call the billiard flow in this infinite domain Lorentz gas and denote it by $G^t$. It preserves $\lambda$, the product
of $\boldsymbol\kappa$ and the counting measure on $\integers^2$.
We assume
{that the scatterer configuration is such that the free flight is bounded (a.k.a. finite horizon 
condition).}

The billiard flow induces a billiard map (or collision map) by the Poincar\'e section taken at collisions. 
Namely, the phase space of the billiard map is
$$
M = \{ (q,v) \in \partial \cD_0 \times \cS^1, \langle v, n \rangle \geq 0\},
$$
where
$n$ is the inward normal vector of $\partial \cD$ at $q$ (that is, $q$ is the point of collision and $v$ is the 
post-collisional velocity). 
The standard coordinates on $M$ are $r$: arc length parameter for $q$ and $\phi$: the angle between 
$n$ and $v$ ($\phi \in [-\pi/2, \pi/2]$ with clockwise orientation).
The billiard map is denoted by $f : M \to M.$  It preserves the invariant measure 
$\nu = c \cos \phi \; dr d\phi$, where $c$ is a normalizing constant. Similarly, the billiard map 
of the Lorentz gas is $T: X \to X$, where $X = M \times \integers^2,$ $T(y,z) = (f(y), z + \tau(y)) $ and
$\tau \in \integers^2$ is the vector connecting the center of the cells where two consecutive collisions take place.
It preserves the invariant measure 
\begin{equation}
\label{eq:sec5mudef}
\mu= \nu \times \mbox{counting}.
\end{equation}

The map $f$ is hyperbolic: there are stable and unstable conefields, 
$\cC^{s}_y, \cC^{u}_y \subset \mathcal T_y M$
such that $Df (\cC^s_y) \subset \cC^s_{f(y)}$,  $Df^{-1} (\cC^u_y) \subset \cC^u_{f^{-1}(y)}$. 
The cones are transversal, that is the angle between any stable vector (an element of $\cC^s_y$ for
some $y$) and any unstable vector is uniformly bounded below by a positive number. 
(In fact there exist some constants $0<c_1<c_2$
so that $\cC^{u}$ can be defined as 
\begin{equation}
\label{eq:uniformcones}
c_1 \leq d\phi / dr \leq c_2
\end{equation} 
$\cC^{s}$ can be defined as $-c_2 \leq d\phi / dr \leq -c_1$ for all $y \in M$.)

The map $f$ is piecewise smooth with singularities at grazing collisions. Furthermore, as the expansion
and the distortion are
unbounded near grazing collisions, it is common to introduce artificial singularities 
$$\mathbb H_k = \{(r, \phi): \phi = \pm \pi/2 \mp k^{-2} \},$$ for $k \geq k_0. $
We call a smooth curve 
{of uniformly bounded curvature} (un)stable if at each point its tangent vector belongs to the (un)stable cone.
An (un)stable curve is homogeneous if 
it does not cross any singularity, genuine or artificial.
We call $W$ a local stable (unstable) manifold if $f^n(W)$ is a stable (unstable) curve for any $n \geq 0$
($n \leq 0$, respectively).

For any unstable curve $W$ and point $y \in W$, we define the Jacobian of $f^n$ on $W$ at $y$ by
$\cJ_W f^n (y) = \| D_xf^n (dy) \| / \| dy \|$ with $dy \in \cT_y W$.
The uniform hyperbolicity implies that there are constants $\Lambda >1$ and $C$ so that 
$\cJ_W f^n (y) \geq C \Lambda^n$ for $n >0$ (and similarly for stable curves and $n < 0$).
Furthermore, after
the above extra partitioning of the phase space, one has the following distortion bounds.
Let $W$ be a homogenenous unstable curve, such that $f^{-n}(W)$ is also homogeneous unstable for $n = 1,...,N-1.$
Then for any $y_1,y_2 \in W$ and $n = 1,...,N-1$ we have
\begin{equation}
\label{BDWu}
e^{-C|W|^{1/3}} \leq \frac{\cJ_W f^{-n} (y_1)}{\cJ_W f^{-n} (y_2)} \leq e^{C|W|^{1/3}}.
\end{equation}
Here, as well as in the sequel, $C$ denotes some finite number depending only on the dynamical
system (and not on the curve $W$ or $n$). Furthermore, the value of $C$ is not important and may change
from line to line.

Given $x\in M,$ the homogenous stable (unstable) manifold of $x$ is the set of points $y$ such that
$f^n y$ and $f^n x$ belong to the same continuity component for all $n\geq 0$ (respectively, for $n\leq 0$). {(Here, in 
the definition of the continuity component,
both genuine and articifial singulairies are accounted for.)}
The homogenous stable (unstable) manifold of $x$ will be denoted by $W^s(x)$ ($W^u(x)$).
It is known that $W^s(x)$ is homogenous stable curve and $W^u(x)$ is homogenous unstable curve.

For any point $y \in M$, we denote by $r_{u}(y)$ ($r_s(y)$) the distance between $y$ and the singularity set,
measured along the unstable (stable) manifold. More generally, given an unstable curve $W$ and $y \in W$,
there is a homogenenous unstable curve $W' \subset f^n(W)$ that contains $f^n(y)$. $W'$ is cut 
by $f^n(y)$ into two pieces, 
the length of the shorter piece is denoted by $r_n(y)$.

The measure of points $y$ such that $r_{u}(y)=0$ or $r_s(y)=0$ is zero. 
It is also true that the measure of points 
having short (un)stable manifolds is small, namely
\begin{equation}
\label{GlobalWsShort}
  \nu (y: \min \{ r_u(y), r_s(y) \} < \eps) \leq C\eps. 
\end{equation}

A pair $\ell=(W, \rho)$ is called a standard pair, if $W$ is a homogeneous
unstable curve and $\rho$ is a probability measure on $W$ satisfying
$$ \bigg| \log \frac{ d \rho}{d \mes}(y_1) - \log \frac{ d \rho}{d \mes}(y_2) \bigg| 
\leq C \frac{|W(y_1,y_2)|}{|W|^{2/3}}, $$
where $|W(y_1,y_2)|$ is the length of the segment of $W$ bounded by $y_1$ and $y_2$.
Here, and also in the sequel, $\mes$ stands for the Lebesgue measure.

The image of a standard pair by the dynamics is a weighted sum of standard pairs
(the image of a homogeneous unstable curve is a family of homogeneous unstable curves
and the regularity of the density of $\rho$ is preserved).
A weighted sum of standard pairs is called a standard family. Namely, a standard family is
a (possibly uncountable) collection of standard pairs 
$\mathcal G=\{ (W_{a}, \nu_a) \}_{ a \in \mathfrak A}$
and a probability measure 
$\eta = \eta_{\mathcal G}$ on $ \mathfrak A$. Such a standard family $\cG$ induces a measure
on $M$ by
\begin{equation}
\label{StFamMes}
 \nu_{\mathcal G} (.) = \int _{\mathfrak{A}} \nu_a (. \cap W_{a}) d \eta_{\mathcal G}(a). 
\end{equation} 

For standard families, the $Z$-function is defined as
\[ \mathcal Z_{\mathcal G} = \sup_{\varepsilon >0} 
\frac{1}{\eps} \int_{\mathfrak A} \nu_a (r_0 < \eps) d \eta_{\mathfrak A}(a). \]

Important special cases are standard pairs ($\mathfrak A$ has a single element
{$\ell$, in which case we simply write $\nu_{\cG} = \nu_{\ell}$}) or 
the decomposition of { the invariant measure}
$\nu$ into conditional measures on unstable manifolds. 
It can be shown that the conditional measures have the required
regularity and the $Z$-function of this family is finite.

Standard pairs are stretched by the dynamics due to expansion and are cut by singularities. The 
next result tells us that "the expansion wins over fragmentation", that is, most 
of the weight is carried by long curves.

\begin{lemma}[Growth Lemma]
\label{LmGrowth}
There are constants $\theta<1, C_1, C_2$ such that for a standard family
$\mathcal G=\{ (W_{a}, \nu_a) \}, a \in \mathfrak A$, and
$\mathcal G_n=  f^n (\mathcal G)$, we have 
\[ \mathcal Z_{\mathcal G_n} < C_1 \theta^n \mathcal Z_{\mathcal G}
+ C_2.\]
\end{lemma}

We also consider standard pairs on the phase space of the Lorentz gas, by shifting $W$ with a 
vector $m \in \integers^2$, where $\ell = (W,\rho)$ is a standard pair for the Sinai billiard. In this case, 
we write $[\ell] = m$.

The Growth Lemma implies 
{that}
for any unstable curve $W$ and for any $n \geq 0$,
$$
\mes (y \in W: r_n(y) < \eps) < C \eps,
$$
where $\mes$ denotes the Lebesgue measure on $W$.

We will also use the following
important consequence of the Growth Lemma ({which is a local version of \eqref{GlobalWsShort}}
see \cite[\S 5.12]{CM06} as well as 
the a proof of \eqref{ManyLongWs} in \S \ref{SSM6}).
Given {an unstable} curve $\gamma$ and a positive number $\delta$, let
$\DS \gamma_\delta=\{x\in\gamma: r_s(x)\geq \delta\}$.
Then there is a constant $K^*$ such that
\begin{equation}
\label{GL-ShMan}
\mes(\gamma-\gamma_\delta)\leq K^* \delta. 
\end{equation}

Another application of the Growth Lemma requires an extra definition. 
Fix a large constant $\brZ.$ In particular we require that $\brZ\geq 2C_2$ where $C_2$ is the constant from
the Growth Lemma. In practice it is convenient to choose $\brZ$ so large that there is a standard family
$\cG$ with $\cZ_\cG<\brZ$ such that $\nu_\cG$ is the invariant measure $\nu.$
We say that a standard family $\cG$ is {\em proper} if $\cZ_\cG\leq \brZ.$ Then the Growth Lemma implies 
that there exists $n_0$ such that for any $n\geq n_0$ and 
for any measure $\brnu$ defined by a proper standard family $\cG$,
the measure 
$\brnu_n(\phi)=\brnu(\phi\circ f^n)$ also corresponds to a proper standard family (namely $f^n \cG$).

Another crucial property of 
partition of $(M, \nu)$ into stable (unstable) manifolds is {\em absolute continuity}.
We refer the reader to \cite[\S 8.6]{BP07} for a comprehensive overview of absolute continuity of stable
and unstable
laminations. Here we just summarize the results for dispersive billiards we are going to use.
Let $W_1$ and $W_2$ be two unstable curves which are close to each other. Let
$$\tW_j=\{x\in W_j: W^s(x)\cap W_{3-j}\} $$
and let $\pi_s: \tW_1\to \tW_2$ be the stable holonomy
$\pi_s(x)=W^s(x)\cap W_2.$ Then $\pi_s$ is absolutely continuous and its Jacobian equals to $J(x, \pi_s x)$
where
(\cite[Equation (5.23)]{CM06}) 
\begin{equation}
  \label{HolJac}
J(x, \pi_s x)=\prod_{n=0}^\infty \frac{\cJ_{f^n W_1}(f^n x)}{\cJ_{f^n W_2} (f^n \pi_s x)}.  
\end{equation}  
Next, \cite[Theorem 5.42]{CM06} tells us that there is a constant $C$ such that
\begin{equation}
\label{HoloBnd}
e^{-C \left(d^{1/3} (x, \pi_s x)+\beta\right)}
\leq J(x, \pi_s x)\leq
e^{C\left(d^{1/3} (x, \pi_s x)+\beta\right)},
\end{equation}
where $\beta$ is the angle between the tangent vector to $W_1$ at $x$ and the tangent vector
to $W_2$ at $\pi_s x.$

A similar statements hold for the unstable holonomy.

Let us list several standard consequences of this fact (\cite{BP07}).

 Given an unstable curve $\gamma$ and a positive number $\delta$, consider the Hopf brush
$
\DS \Lambda_\delta=\bigcup_{x\in \gamma_\delta} W^s(x).$ 
Consider the measure $\hnu$ defined by
$$\DS \hnu(A)=\int_{\gamma_\delta} \mes_{W^s} \left(W^s(x)\cap A\right) 
d \mes_{\gamma} (x).$$ 
Let $\nu_{\Lambda_\delta}$ denote the restriction of
$\nu$ to $\Lambda_\delta.$ 
Suppose that $|\gamma|\geq 2 K^*\delta$ so that \eqref{GL-ShMan} implies 
that $\Lambda_\delta\neq\emptyset.$ 
Then there is a constant $\kappa_1=\kappa_1( \delta)$ such that 
\begin{equation}
  \label{AC3}
\kappa_1\leq 
\frac{d\hnu}{d\nu_{\Lambda_\delta}}\leq \kappa_1^{-1}. 
\end{equation}
From the foregoing discussion it is not difficult to see that
 there is a constant $\kappa_2=\kappa_2(\delta)$ such that for each $\gamma$ of length
 at least $2 K^*\delta$,
\begin{equation}
  \label{AC4}
\nu(\Lambda_\delta(\gamma))\geq \kappa_2. 
\end{equation}   
 
Another consequence of \eqref{AC3} is that if $A$ is a set of measure zero, then
\begin{equation}
\label{AC2}
\text{for $\nu$ almost every} \;x,\; \mes(W^s(x)\cap A)=\mes(W^u(x)\cap A)=0.
\end{equation}

{
We finish this section by commenting on the case of unbounded free flight
(infinite horizon). The preliminaries discussed in this section extend to that case, too.
The billiard map is local-global and global-global mixing just like in the case of finite horizon
 (see Section
\ref{SSLG}) as the MLLT holds with scaling $L_n = \sqrt{n \log n}$ \cite{SzV07}. We have little doubt 
that the
same holds in continuous time, too, but we are not aware of any explicit proof of the MLLT
in the literature. To study the perturbed models as in \S\S
\ref{SSSLocPert}--\ref{sec:fields} one would need a more serious departure from the
case of finite horizon (but see \cite{CD09b, PT20} for some results in these directions). In the
rest of this paper, we
only study the case of finite horizon.
}


\section{Examples}
\label{secEx}
Here we describe several examples satisfying the assumptions of Section \ref{ScResults}.
{Each time we use the MLLT or its variants (shifted MLLT,
AMLLT), we choose $L_n = \sqrt n$ and,  unless noted otherwise, $\fp$ a centered Gaussian density.
We formulated the results of Section 
\ref{ScResults} with general 
$L_n$ and $\fp$ because there are other natural examples 
(e.g. the infinite horizon Lorentz gas or interacting particle systems studied e.g. in 
\cite{PGSz12})
whose global mixing properties could be approachable by our methods.
}

\subsection{Lorentz gas}
\label{SSLG}

The mixing local
limit theorem holds for Lorentz gas with finite horizon
in both discrete \cite{SzV04} and continuous setting \cite{DN16}.
Accordingly Theorem \ref{LLTIMMIx} applies to both Lorentz collision map and Lorentz flow,
and so, both systems enjoy both local global {
mixing with respect to $\bbG_{O}$} and global global mixing with respect to 
$ \bbG_{AO}.$ 

One can also consider a Lorentz tube, where instead of motion on the plane the particle moves
on the strip with a periodic configuration of convex scatterers removed. As before 
\cite{SzV04, DN16} give MLLT in both discrete and continuous setting and so the system
enjoys both local global and global global mixing with respect to 
$\bbG_O.$

\subsection{Local Perturbations of Lorentz gas.}
\label{SSSLocPert}
Consider a billiard in a domain which is periodic 
outside of some ball.
If the limiting periodic configuration has finite horizon
(or equivalently, the perturbed configuration has finite horizon)
then the conditions of Propositions \ref{PrLocPert} 
and \ref{PrLocPertFlow} are satisfied and so the system enjoys global global 
mixing. On the other hand, local perturbations of the Lorentz gas do not have to be local 
global mixing. 
Indeed, we can trap particles in a bounded part of the phase space. For example, 
by allowing non-convex scatterers, one can arrange that the system
has a stable elliptic orbit, so that the set $\cB$  of bounded orbits has positive measure.
Let $\cB_L$ be the set of orbits which always stay within distance $L$ from the origin. Take 
$\phi$ such that $\int_{\cB_L} \phi d\mu>0.$ Take two functions $\Phi_1, \Phi_2\in \bbG$ 
such that 

(i) $\Phi_2>\Phi_1$ and moreover

(ii) $\Phi_2-\Phi_1\geq 1$ inside the ball of radius $L;$ 

(iii) $\brPhi_2=\brPhi_1.$ 

In this case
$$ \int \phi[(\Phi_2-\Phi_1)\circ \tT^n] d \mu\geq \int_{\cB_L} \phi d\mu $$
does not tend to 0, so it is impossible that both
$$ \int \phi(x) \Phi_2(\tT^n x)d\mu(x)\to \mu(\phi)\brPhi_2
\text{ and }
\int \phi(x) \Phi_1(\tT^n x)d\mu(x)\to \mu(\phi)\brPhi_1. $$

However, the system remains local global mixing if the configuration is a finite
perturbation (i.e. finitely many scatterers discarded, finitely many new ones included)
of a periodic Lorentz gas such that the scatterers in the
entire configuration (including the perturbed part) are strictly convex, disjoint and have $\mathcal C^3$
boundary. We call such a perturbation a {\em mild perturbation}. 
{ Without loss of generality, we can assume that the fundamental domain is large enough so
that outside the cell at the origin, the system is periodic. Thus we are in the setup of \S \ref{sec:AMLLT}, with $d_1 = 0$, $d_2 = 2$, $\cB = \{ 0 \}$, 
$M$ the phase space of the billiard map on any cell but zero, $D_0$
the phase space of the billiard map in the zeroth cell and
the measures $\nu$ and $\nu_0$ are the usual  
measures on $M$ and $D_0$, as defined in Section \ref{BilPrel}
(in condinuous time, we need to define $M$ and $D_0$ as the phase space of the flow, restricted to the
same cells as before and consider the invariant physical measures on them, denoted by 
$\boldsymbol\kappa$ in Section
\ref{BilPrel}).

Mildly perturbed Lorentz gases
are  local global mixing with respect to $\bbG_O$ 
and global global mixing with respect to
$ \bbG_{AO}$ as implied
by Theorem \ref{ThALLT-Mix} and  the following.}
\begin{theorem}
\label{thm:ggLorentzpert}
The mildly perturbed periodic Lorentz gas satisfies the AMLLT.
\end{theorem}
\begin{proof}

 The proof is similar to (but easier than) the proof of Proposition 3.8 in \cite{DN16} 
 so we provide only a sketch of the argument. 
 
We begin with discrete time. In the proof we will use letters with
tildes to denote the objects
associated to the mildly perturbed Lorentz gas, and the same letter without tildes will
refer to periodic (unperturbed) system.

{ Let $\brnu_{\phi, w}$ be the measure defined by either \eqref{DefNuZ} or 
\eqref{DefNuZB}}.
The {\em global} central limit theorem for mildly perturbed periodic Lorentz gas is proved in
\cite[Theorem 1]{DSzV09}.  Thus there is a positive definite matrix $D$ such that 
$$ 
 \brnu_{ \phi, w}
\left(\frac{\ttau_n}{\sqrt{n}}\in \Omega
+ \frac{w}{\sqrt n}\right)\to 
\nu(\phi)
\iint_\Omega \fg (u) du $$
\noindent
{ as $n \to \infty$},
where $\fg$ is the density of the centered Gaussian distribution with covariance matrix $D$
and $\Omega \subset \reals^2$ is a set whose boundary has zero Lebesgue measure
{ and the convergence is uniform for $\phi$
with bounded Lipschitz norm.}

We need to evaluate
$$ I_n=\brnu_{ \phi, w}
\left(\psi(\tx_n) 1_{\ttau_n=
\lfloor  \bz\sqrt n  \rfloor  - w
}\right).$$
{To simplify the notation, we drop the subscript of
$\brnu$ and write $z_n =  \lfloor  \bz\sqrt n  \rfloor  - w$.}
Take $\delta_t\ll 1$ and denote $n_2=\delta_t n,$ $n_1=n-n_2.$

Let the measure $\nu^{\brz}$ be the normalized version of the restriction
of $\tT^{n_1*} \brnu$ to the cell $\brz$. That is, if
$p_{n_1}(\brz)=\brnu(z \circ \tT^{n_1} =\brz)$ and $A\subset M$, then 
$$
\nu^{\brz} (A) = \frac{1}{p_{n_1}(\brz)} \brnu \left(\tx: \tT^{n_1} (\tx) \in \left(A \times  \{ z = \brz\} \right)\right).
$$
Then we have the decomposition
$$
I_n=\sum_{\brz\in \integers^2 - \{ 0\}} p_{n_1}(\brz) \nu^\brz(\psi(\tx_{n_2})1_{\ttau_{n_2}=z_n-\brz}) + \heps_1
$$
where $\heps_1$ is an error term corresponding to the set of points $\tx$ so that $z \circ \tT^{n_1}(\tx) = 0$
and we assumed that all perturbations are in the zeroth cell.

Choose $K\gg 1$ and consider the following approximation
\begin{equation}
\label{Viable}
 I_n=\sum_{|\brz-z_n|\leq K\sqrt{n_2}} p_{n_1}(\brz) \nu^\brz(\psi(x_{n_2})1_{\tau_{n_2}=z_n-\brz})+
\heps_1 + \heps_2
\end{equation}
where $\heps_2$ is an error term. Note that there are no tildes inside 
$\nu^\brz(\cdot).$ That is we pretend that the particle moves in the unperturbed environment
for the last $n_2$ collisions. The error $\heps=\heps_1+\heps_2$ comes from two sources:

(A) There is a contributions from the cells with $|\brz-z_n|>K\sqrt{n_2}$ and

(B) the particle may visit the perturbed region for some $k \in [n_1, n]$.

Given $\eps$ we can choose $\delta_t$ so small and $K$ so large that 
both (A) and (B) 
have contributions which is less than $\frac{\eps}{n}$ similarly to \cite[\S 6.2]{DN16}. Note that
\cite[Lemma 2.8(b)]{DN16}, which is extensively used in this step, is formulated for the Lorentz
tube and thus is not directly applicable here. 
However, we can replace it by \cite[Lemma 4.8(b)]{DN17}, which is valid in a much 
more general setting, including the Lorentz gas. 

Returning to the main term in \eqref{Viable} we can use the MLLT for the periodic Lorentz gas to conclude that
\begin{equation}
\label{LLT-Last}
 \nu^\brz(\psi(x_{n_2}) 1_{\tau_{n_2}=z_n-\brz})\approx
\frac{1}{n_2} \fg \left(\frac{z_n-\brz}{\sqrt{n_2}}\right)\nu(\psi). 
\end{equation}
Let us divide the set 
$\{z: |z-z_n| \leq K \sqrt{n_2}\}$ into boxes $B_j$ of size $\delta_s \sqrt{n}$ where $\delta_s\ll \delta_t.$
Then,
$$ \sum_{|\brz-z_n|\leq {K} \sqrt{n_2}}  p_{n_1}(\brz) \nu^\brz(\psi(x_{n_2})1_{\tau_{n_2}=z_n-\brz})$$
\begin{equation}
\label{Mesh}
  \approx
\frac{\nu(\psi)}{\delta_t n}  \sum_j \sum_{\brz\in B_j} p_{n_1}(\brz) \fg\left(\frac{\brz-z_n}{\sqrt{n_2}}\right). 
\end{equation}
Since the oscillation of $\DS \fg\left(\frac{\brz-z_n}{\sqrt{n_2}}\right)$ on $B_j$ is small, 
we can replace 
it by
$\DS \fg\left(\frac{z^{(j)}-z_n}{\sqrt{n_2}}\right)$ where $z^{(j)}$ is the center of $B_j.$
Accordingly
$$ \sum_{\brz\in B_j} p_{n_1}(\brz) \fg\left(\frac{\brz-z_n}{\sqrt{n_2}}\right)\approx
\fg\left(\frac{z^{(j)}-z_n}{\sqrt{n_2}}\right) \sum_{\brz\in B_j} p_{n_1}(\brz)=$$
\begin{equation}
\label{BoxSum}
\fg\left(\frac{z^{(j)}-z_n}{\sqrt{n_2}}\right) \brnu(\ttau_{n_1}\in B_j). 
\end{equation}
The global CLT for the mildly perturbed Lorentz gas and the fact that $z^{(j)}$ are close to $z_n$ for all $j$
imply that
\begin{equation}
\label{BoxCLT}
 \brnu(\ttau_{n_1}\in B_j)\approx \delta_s^2 \fg(\bz) 
\end{equation} 
Combining \eqref{Viable}--\eqref{BoxCLT} we obtain
$$ I_n=\frac{\fg(\bz) \nu(\psi)}{n} \sum_j \frac{\delta_s^2}{\delta_t} 
 \fg\left(\frac{z^{(j)}-z_n}{\sqrt{n_2}}\right) . $$
The last sum is the Riemann sum of the integral of a Gaussian density over the set $\{|z|<K\}.$
Accordingly taking $K$ large and choosing $\delta_s$ small to make the mesh sufficiently fine,
we can make the last sum as close to 1 as we wish. This completes the sketch of 
proof of the AMLLT in the discrete time case.

The continuous time case is similar but we need to use the MLLT for flows proven in \cite{DN17}.
\end{proof}

\subsection{Lorenz gas in a half strip.} 
Consider a Lorentz gas in a half strip, i.e. in $\reals^+\times [0,1]$ with a periodic configuration of convex
scatterers removed. (By periodicity we mean that if $\cS$ is a scatterer in our configuration and
$\cS_{\pm}:=\cS \pm (1,0)$,
then $\cS_+$ is in the scatterer configuration and if 
$\DS \cS_- \subset (\reals^+\times [0,1])$, then $\cS_-$ also belongs to the configuration).

{ 
Similarly to the mildly perturbed Lorentz gas, we are in the setup of \S \ref{sec:AMLLT}, 
now with $d_1 = 1$, $d_2 = 0$, $\cB = \{ 1 \}$.}
Using \cite[Theorem 2]{DSzV09} and proceeding as in the proof of Theorem \ref{thm:ggLorentzpert}, we have

\begin{theorem}
Lorentz gases in half strips  satisfy the AMLLT with
{  
$\fp$ being the probability density of the absolute value of a centered Gaussian random variable. }
\end{theorem}

{ Thus by Theorem \ref{ThALLT-Mix}, the Lorentz gas in a half strip satisfies both 
 local global mixing with respect to $\bbG_O$ 
and global global mixing with respect to
$ \bbG_{AO}.$}

\subsection{Lorenz gas in a half plane.} 
{
Consider a Lorentz gas in a half plane, i.e. in $\reals^+\times \reals$ with a periodic configuration of convex
scatterers removed. (By periodicity we mean that if $\cS$ is a scatterer in our configuration, then
$\cS + (1,0)$, $\cS \pm (0,1)$ are also in the configuration. If  
$\DS \cS - (1,0) \subset (\reals^+\times \reals)$, then $\cS - (1,0)$ also belongs to the configuration).

Similarly to the mildly perturbed Lorentz gas and to the Lorentz gas in a half strip, 
we are in the setup of Section 
\ref{sec:AMLLT}, now with $d_1 = 1$, $d_2 = 1$, $\cB = \{ 1 \} \times \integers$.
Using \cite[Theorem 4]{DSzV09} and proceeding as in the proof of Theorem \ref{thm:ggLorentzpert}, we have

\begin{theorem}
Lorentz gases in the half plane satisfy the AMLLT with
$\fp$ being the density at time 1 of the Brownian motion with diffusion matrix of the Lorentz process
reflected from the $y$ axis.
\end{theorem}

Thus by Theorem \ref{ThALLT-Mix}, the Lorentz gas in a half plane satisfies both local global mixing with respect to $\bbG_O$ 
and global global mixing with respect to
$ \bbG_{AO}.$
}

\subsection{Lorentz gas with external fields}
\label{sec:fields}

\subsubsection{Lorentz gas in asymptotically vanishing potential fields}
\label{ex:vanpotential}

In this example we consider the same configuration of scatterers as in Example \ref{SSLG}
but assume that the motion between collisions is subject to the potential
$$ \ddot{q}=-\nabla U .$$
We suppose that the first three derivatives of $U$ are uniformly bounded and 
that
\begin{equation}
\label{VanishPot}
\lim_{|q|\to\infty} U(q)=0, \quad \lim_{|q|\to\infty} \nabla U(q)=0 .
\end{equation}
An example of such system is given by the Coulomb potential 
\begin{equation}
\label{EqCou}
 U(q)=\frac{\mathbf{e}}{|q|}. 
\end{equation} 
For the Coulomb potential it is natural to assume that the origin is contained in the center of one of the scatterers. In this case $U$ is bounded.

In any case our system is Hamiltonian 
preserving the energy $H={\frac12} v^2+U(q).$ 
{ 
Sinai billiards with external fields were studied in \cite{Ch01, C08}. 
First, note that
the phase space of both the map and the flow is the same as in case of no external field.
Next, we note that the flow $\tG$ preserves the
Lebesgue measure and the collision map $\tT$ preserves the measure $\mu$ defined in 
\eqref{eq:sec5mudef} (see e.g. the Remark on page 201 of \cite{Ch01}).

\begin{theorem}
Under assumption
\eqref{VanishPot} both the collision map $\tT$ and the continuous time system $\tG^t$
enjoy global global mixing with respect to
$\bbG_{AO}.$    
\end{theorem}

\begin{proof}
We claim that both $\tT$ and $\tG^t$ are 
very well approximated by the Lorentz gas
and so by
Theorems \ref{ThInfAppGGM} and \ref{ThInfAppGGMFlow}
the result will follow.
To prove the above claim, it is sufficient to check condition (i)
of Definition \ref{DefWApp} (and its continuous time counterpart). In continuous time,
we can choose $A_{z,\eps} = \emptyset$ as the flow $\tG^t$
is continuous and for $R$ large,
is uniformly close to the unperturbed billiard flow $G^t$ up to time $1$ by condition
\eqref{VanishPot}. To check condition (i) for the map, choose $A_{z,\eps}$ as the 
$\delta$ neighborhood of the primary singularity set of the unperturbed billiard map $T$.
By choosing $\delta$ sufficiently small, we clearly have $\mu(A_{z,\eps}) < \eps$ and now choosing
$R$ large (and consequently the field small), we have \eqref{eq:vwapproxcond}.
\end{proof}

Similarly to \S \ref{SSSLocPert}}, the assumption \eqref{VanishPot} is insufficient to ensure hyperbolicity close to the origin.
In particular the system could have elliptic islands in the bounded part of the space (cf. \cite{RKT})
and so it may fail to be local global mixing. 
On the other hand,  
{our next result gives local global mixing under the extra assumption that the 
field is small everywhere.

\begin{theorem}
\label{thm:fieldlg}
Assume besides \eqref{VanishPot} that $||U||_{C^3}$
is sufficently small (e.g. in the Coulomb potential case the charge $\be$ is small). Then
both the collision map $\tT$ and the continuous time system $\tG^t$
enjoy local global mixing with respect to
$\bbG_U.$    
\end{theorem}
}

\begin{proof}
{By Theorem \ref{ThLGM-App}, 
it suffices to check conditions (M1)-(M6).

We begin with the discrete time system. 
Much of the theory discussed in Section \ref{BilPrel}
has been extended to the Sinai billiards on compact phase space
with external fields in \cite{Ch01, C08}. Several of these results can be used in
our non-compact setup, too, since the proofs do not rely on the compactness of the phase
space.
For example, standard pairs are defined in \cite{C08}.
In fact, standard pairs for $\tT$ are exactly the same as standard pairs for $T$ (of course,
unstable manifolds are different but the unstable cone can be chosen the same).
Using the notation of Section~\ref{BilPrel}, we say that a standard family is compactly
supported if there is a finite set $A \subset \integers^2$ so that for all standard pairs $\ell$
in the family, $[\ell] \in A$.}

Let $\fM$ to be the set of all compactly supported {proper} standard families. 
Specifically, we require that $\fm\in\fM$ satisfies
\begin{equation}
\label{GrowthC-K}
\fm(x: r(x)<\eps)\leq K \eps,
\end{equation}
where $K$ is a sufficiently large constant { only depending on the system}. Then
(M1) is checked in \cite{C08}. To check (M2), let $\phi$  be a Lipschitz function supported on a single scatterer 
 $\Omega$. (Note that it suffices
to check the local global mixing for Lipschitz functions $\phi$ as the set of Lipschitz functions is
dense in $\mathbb L$. The condition that $\phi$ is supported on a single scatter is also not restrictive
since a function supported on a finite set of scatterers is a finite linear combination of functions
supported on a single scatterer.) 
We first observe that for each $\delta$ there exists $K(\delta)$ such that
if $\phi$ has the following properties:
\begin{equation}
\label{ProductProper}
  \delta\leq \phi\leq \delta^{-1}, \quad \mu(\phi)=1, \quad \text{Lip}(\phi)\leq 2, 
\end{equation}
then $\phi\mu\in \fM$ where $\fM$ is defined by \eqref{GrowthC-K} with $K=K(\delta),$
see e.g. {\cite[Proposition 5.6]{Ch01}}. Pick a large $R{\gg \delta^{-1}}$
We have the following decomposition:
$\phi={R} \one_\Omega-({R}-\phi) \one_\Omega.$ Thus $\phi=c_1\phi_1-c_2\phi_2$ where
$c_1$ and $c_2$ are constants and
\begin{equation}
\label{DifDecomp}
  \phi_1=\frac{\one_\Omega}{\mu(\Omega)}, \quad \phi_2=\frac{\one_\Omega-
{\frac{\phi}{R  }}}
      {\mu(\Omega)-{\frac{1}{R}}}.
\end{equation}      
Note that as $R\to\infty$, $\phi_2\to 1_\Omega / \mu(\Omega)$
in the space of Lipschitz functions, so if $R$ is sufficiently large
then $\phi_1, \phi_2$ satisfy \eqref{ProductProper} 
with constant $\delta$ depending only on the minimal perimeter of
the scatterers in our configuration. By the foregoing discussion,
$\phi_1 \mu,\phi_2 \mu\in \fM.$

{To prove (M3), 
we use the transversality of the unstable curves to singularities of the system
(see \cite[Section 4.5]{CD-BBM} for a similar argument). Specifically, given $\eps$ and $n$, we choose some
$\delta \ll \eps$. Then for the given $\eps, n, \delta$, we choose $R$ so large so that for every $x$
with $|z(x)| > R$ and for any $s \in [0,n (\tau_{\max}+1)]$, $d(G^s(x), \tG^s(x)) < \delta$. Such 
an $R$ exists since for small field, the trajectories are uniformly close to the unperturbed ones
(here, $\tau_{\max}$ is the maximum free flight time
 of the unperturbed system and consequently the
maximum free flight time of the perturbed system
is bounded by $\tau_{\max} +1$.) Thus choosing $\delta$ small, we can ensure
that
the singularity curves
of $\tT^n$ are in the $\eps^2$ neighborhood of those of $T^n$. 
Furthermore, the singularity curves
of $\tT^n$ are transversal to the unstable cones by
\cite[Lemma 3.10]{Ch01}. 
Let $\fm \in \fM$, $\ell = (W,\rho)$ a standard pair in $\fm$ and $x \in W$. If
$|z(x)| > R$ and $d(T^nx, \tT^nx) \geq \eps$, then by the foregoing discussion, $x$
is necessarily $C \eps^2$ close to an endpoint of $W$ (here $C$ is a geometric constant 
coming from 
the transversality). By \eqref{GrowthC-K}, the $\fm$ measure of such points is bounded
by $KC \eps^2$. For $\eps$ small enough, $KC\eps^2 < \eps$ and so (M3) follows (clearly, it 
is sufficient to prove (M3) for $\eps$ small enough).}
 
Next, let $\fM_\eps$ be the set of standard families on $M$ such that 
all standard pairs in $\fm$ is longer than $\eps.$ The local limit theorem for standard families
follows from {the mixing LLT for $T$ \cite[Lemma 2.8]{DN16}}. Thus (M4) holds.

{
Next, in our system a stronger variant of (M5) holds, namely $n_0$ is uniform in $\fm \in \fM$. Indeed,}
for $\fm$ in $\fM$ let $\fm_j'$ is the measure corresponding to the standard pairs
from ${\tT^n} \fm$ which belong to $\{z=j\}$ and have length greater than $\eps$.
{The desired inequality of (M5) follows from the growth lemma 
(see \cite[Lemma 5.3]{Ch01} and the discussion on page 95 of \cite{C08}).}

Since checking (M6) requires more effort, we postpone it to Section~\ref{sec:m6}.

The continuous time case can be handled similarly. We refer the reader to 
\cite{DN16e, BNSzT}
for the Growth Lemma and related results in the continuous time setting.
\end{proof}

\subsubsection{Lorentz gas in external field and Gaussian thermostat.} 
Suppose that the system moves in the same domain as
the Lorentz gas but the motion between the collisions is not free but rather satisfies
$$ \ddot{q}=E(q)-\frac{\langle \dot{q}, E(q)\rangle}{||\dot{q}||^2}$$
where $E(q)$ is a periodic field and the second term models energy dissipation.
This system is a $\integers^2$-cover of a Sinai billiard in external field which we will denote by $f$.
{There are two important differences between this model and the one 
studied in \S \ref{ex:vanpotential}:
this one is easier in the sense that it is periodic but more difficult in the sense that the 
Lebesgue measure is no longer invariant.
However, \cite{Ch01} implies that} $f$ has unique SRB measure $ \mu_E$
if $||E||_{C^1}$ is sufficiently small.
Furthermore, a Young tower can be constructed 
by the results of \cite{Ch01,C08} (see also \cite{CWZ}). Thus the (shifted) MLLT holds for 
$ (f, \mu_E)$
{by \cite[Lemma 4.3]{DN17}}
The shifted MLLT for continuous time system also follows from \cite[Theorem 4.1]{DN17}.
Accordingly by Theorem \ref{SLLTIMMIx}, we have local global and global global mixing 
with respect to $(\bbL, \bbG_U).$ 
We note that for typical $E$ (including the constant field) 
the drift in the CLT is not equal to zero (\cite{CELS93}).
We also note that in the presence of the drift, the system
is dissipative in the sense of ergodic theory, that is, almost every particle tends to infinity.
This gives a physical example of a system which enjoys both local global and  global global
mixing but is not ergodic.

\subsection{Galton board.}
This model is similar to Example \ref{ex:vanpotential}, however, we do not assume that
the potential is vanishing at infinity. Namely we consider a particle moving in a half plane
$q_1>0$
with a periodic configuration of convex scatterers removed
(we confine the particle to the half plane  by 
adding the vertical axis $q_1=0$ to the boundary of our domain).
The motion between collisions 
is subject to
a constant force field
which corresponds to a linear potential $U=-\mathbf{g} q_1.$ This system preserves the energy
$$H=v^2/2-\mathbf{g} q_1.$$
It is convenient to use the following coordinates: $q\in\reals^2$ is the position of the particle and
$\theta$ is the polar angle of the velocity vector $\tan \theta= \dot{q}_1/\dot{q}_2.$
Then the speed could be recovered using the equation $|v|=\sqrt{2(H+\bg q_1)}.$
{ In Lemma \ref{LmGB-LG} below we will see that the evolution of $q$ and $\theta$ coordinates 
is well approximated by the Lorentz gas. Therefore the appropriate space of observables
are functions which are uniformly continuous in $(q, \theta)$ coordinates and admit the averages 
on large cubes. Namely given $\fq=(\fq_1, \fq_2)\in  [0, \infty) \times \reals$ and $R>0$ such that $\fq_1>R$ 
consider the cube
$\DS \Omega_{\fq, R}=\{(q,\theta): |q-\fq|_\infty\leq R\}$ 
and let 

${\bbG_U}=\{\Phi: \Phi$ is uniformly continuous in $(q, \theta)$ variables and for each $\eps$ there is $R_0$
such that if $R\geq R_0$ then for each $\Omega_{\fq, R}$ as above
$$\left. \left|\frac{1}{\mu(\Omega_{\fq, R})} 
\int_{\Omega_{\fq, R}} \Phi(q, \theta) d\mu-\brPhi\right|\leq \eps\right\} .$$}
\smallskip

The main result of this section is

\begin{theorem}
\label{ThGlobMixGalton}
  There exists $H_0$ such that if $H\geq H_0$, then both the collision map $\tT$ and the continuous flow
  $\tG^t$ enjoy global global mixing { with respect to 
$\bbG_{AO}$} and local global mixing 
  { with respect to $\bbG_U$.}
\end{theorem}

{ In order to prove Theorem \ref{ThGlobMixGalton} we need to recall several
results from~\cite{CD09}. 

\begin{lemma}
\label{LmGB-LG}
The collision map $\tT$ for Galton board is well approximated for large kinetic energy 
 by the collision map $T$
of the Lorentz gas. More precisely, the following condition holds

$\overline{(M3)}$
For each $\eps>0$ and $n\in \naturals$ there exists $R>0$ such that if  $\fm$ is
a measure corresponding\footnote{in the sense of \eqref{StFamMes}} to a proper standard family, then
$$  \fm(x: q_1(x)\geq R\text{ and } d(T^n x, \tT^n x)\geq \eps)\leq \eps. $$
\end{lemma}

Note that the condition $\overline{(M3)}$ above is different from the condition (M3) imposed 
in Section \ref{ScResults}. Namely, we replace the requirement \\$q_1^2+q_2^2\geq R^2$ by a 
stronger requirement $q_1>R.$ Lemma \ref{LmGB-LG} is proven in \cite[Section 3]{CD09}, however we recall the argument since it plays
an important role in the analysis below.

\begin{proof}
Let $(q_n, \theta_n, K_n)$ denote the position, direction and kinetic energy of the Galton
particle after $n$ collisions. The motion until the next collision is obtained by solving the following ODE
$$ \frac{d^2 q}{d t^2}=\bg e_1, \quad q(0)=q_n, \quad 
\frac{dq}{dt} (0)=\sqrt{2K_n} (\cos \theta_n, \sin \theta_n). $$
Making the time change 
\begin{equation}
\label{RescaledTime}
s=\frac{t}{\sqrt{2K_n}}
\end{equation}
 (note that changing the time does not change the place of the next collision) we get
\begin{equation}
\label{GB-Rescaled}
\frac{d^2 q}{d t^2}=\frac{\bg}{2K_n}  e_1, \quad q(0)=q_n, \quad 
\frac{dq}{dt} (0)=(\cos \theta_n, \sin \theta_n). 
\end{equation}
Note that $K_n=H+\bg (q_n)_1$, where $H$ is the particle's energy.
Therefore by taking $R$ large enough we can make the RHS of
 the ODE in
\eqref{GB-Rescaled} 
as small as we wish if $(q_n)_1\geq R.$ 
Accordingly the solution to \eqref{GB-Rescaled} can be made as close as we wish to 
the solution of 
$$ \frac{d^2 q}{d t^2}=0, \quad q(0)=q_n, \quad 
\frac{dq}{dt} (0)=(\cos \theta_n, \sin \theta_n). $$
Since the last equation describes the flow of the Lorentz gas without external field
between two collisions,
the lemma follows.
\end{proof} 
 
Since the Lorentz gas is hyperbolic, we have that the Galton board dynamics is
also hyperbolic for large kinetic energies. The condition that the total energy is large ensures 
that the kinetic energy is large as well, so the hyperbolicity persists in all of the phase space.

\begin{proposition}
\label{PrBessel}
There are constants $\sigma$ and $\brsigma$ such that the following holds.

Suppose that $(q(0), v(0))$ is distributed according to some standard family.

(a) Let $K_n$ denote the kinetic energy of the particle after $n$ collisions.
Then the random process $\cK^n(t)=\frac{1}{\sqrt{n}} K_{tn}$ converges in law,
as $n\to \infty$ to $\cK(t)$ which is 
the solution to the following stochastic differential equation:
\begin{equation}
\label{K-SDE}
 d\cK=\frac{\brsigma^2}{4\cK} dt+\brsigma d\cW, \quad \cK(0)=0.
\end{equation} 

(b) Let $K(t)$ denote the kinetic energy of the particle at time $t$. Then the random process
$\DS \hat\cK_T(t)=\frac{K(t T)}{T^{2/3}}$ converges in law, as $T\to \infty$ to $\hat\cK(t)$ 
which is the solution to the following stochastic differential equation:
\begin{equation}
\label{K-SDE-cont}
 d\hat\cK=\frac{\sigma^2}{2\sqrt{2\hat\cK}} dt+(2\hat\cK)^{1/4} \sigma d\cW, \quad \hat\cK(0)=0.
\end{equation}
\end{proposition} 
}
 Note that the equations \eqref{K-SDE} and \eqref{K-SDE-cont}
 are well posed despite the singular coefficients as discussed
in \cite{CD09}.
\begin{proof}
{ Part (b) is a restatement of Theorem 3 in \cite{CD09}. Namely \cite{CD09} 
uses the rescaled time $s=\frac{t}{T^{1/3}}$ (cf. \eqref{RescaledTime}).
In the rescaled time the part (b) states that
$\frac{K(s T^{4/3})}{T^{2/3}}\Rightarrow \hat\cK(t)$ as $T\to\infty.$ 
Denoting $\eps=T^{-2/3}$ we can rewrite the last statement as  
$\eps K(s \eps^{-2}) \Rightarrow \hat\cK(t)$ as $\eps\to 0$ 
which exactly the statement of Theorem 3 in \cite{CD09}.

Next we discuss the part (a).
In the case we start away from $0$ and the process $\cK^n$
is stopped when it reaches too high or too low values, \eqref{K-SDE} is proven in \cite[Theorem 4]{CD09}. 
The removal of those cutoffs can be done in the same way as in the continuous time case, see the proof of
Theorem 3 in \cite{CD09} (note that this theorem assumes that the total energy $H$ is large
enough). 
}
\end{proof}

{ We mention that the explicit formulas for $\sigma$ and $\brsigma$ are the following
(cf. \cite[page 839]{CD09}). 
Let $\tsigma$ be the diffusion coefficient of $q_1$ for the Lorentz gas with respect to the
discrete time. That is 
$$\tsigma^2=\lim_{n\to\infty} \nu\left(\frac{(q_{0,n})_1^2}{n}\right)$$
where $q_{0,n}$ is the position of the particle after the $n$-th collision 
in the Lorentz gas and $\nu$ is any 
smooth compactly supported measure. Then
$ \brsigma=\tsigma \bg$ and $\sigma=\brsigma/\sqrt{\brtau}$ where $\brtau$
is the free path length. However, we do not need the explicit values of $\sigma$ and $\brsigma$
in the proof of Theorem \ref{ThGlobMixGalton}. }

\begin{proof}[Proof of Theorem \ref{ThGlobMixGalton}]
{ Given the background presented above, the proof  proceeds similarly to 
the arguments of Section \ref{ScProofs} with minor modifications described below.}
\smallskip

{
{\bf Global global mixing for $\tT$.}
Given Lemma \ref{LmGB-LG}, 
the proof of the global global mixing is the same as the proof of Theorem \ref{ThInfAppGGM}
(a) with $d_1 = d_2 = 1$, except instead of the fact that $z(\tT^k)$ 
is large for all $k\leq n$ for most initial conditions in our cube, 
we use that $q_1(T^k x)$ (and, hence, $K(T^k x)$)
is large for all $k\leq n$ for most initial conditions in our cube.}
\vskip2mm

{\bf Local global mixing for $\tT$.}
We check (slightly modified) conditions (M1)--(M6). We choose $\fM$ and $\fM_\eps$
in the same way as in Example \ref{ex:vanpotential}. 
(M2) and (M4) are checked in the same way as in
that example. (M1) and (M5) follow from \cite[Lemma 2.1]{CD09}.
{ We already checked
$\overline{(M3)}$, which is an analogue of (M3), in Lemma \ref{LmGB-LG}.
Since $\overline{(M3)}$ is weaker than (M3), we need to replace (M6) by a stronger condition,
namely\vskip1mm

$\overline{(M6)}$
For each $\fm\in \fM$ and for each $R>0$, $\fm(|K(\tT^n x)|\leq R)\to 0$ as
$n\to\infty$ where $K$ denotes the kinetic energy.} \vskip1mm

Similarly to Theorem \ref{ThLGM-App}, local global mixing is implied by
(M1), (M2) $\overline{(M3)}$, (M4), (M5) $\overline{(M6)}$. It remains to verify $\overline{(M6)}$. To this end,
we note that by Proposition \ref{PrBessel}(a), 
$\frac{K_n}{\sqrt{n}}$ converges to $\cK(1)$, where $\cK(\cdot)$ is the solution to \eqref{K-SDE}.
Note that $\cK(t)$ is a power of
the square Bessel process, so its density can be computed explicitly (cf. \cite{D08}). 
In particular,
$\Prob(\cZ=0)=0$ proving $\overline{(M6)}.$
\vskip2mm

{\bf Local global mixing for $\tG^t.$}
In this case, we also
need to modify (M1)--(M6).
Note that if $q(t)\sim Q\gg 1$, then $v(t)\sim \sqrt{Q}$ so the particle will travel distance of order
$\sqrt{Q}$ during a unit time interval. This distance is too large for Lorentz particle to serve as a good approximation to the Galton particle. The good news is that a much shorter time is sufficient to observe the 
LLT on Galton board. 

{ Note that Lemma \ref{LmGB-LG}  does not tell us that $\tG^t$ is well approximated by $G^t.$
Instead $G^t$ approximates the rescaled flow. Namely, let $\hG^s$ be obtained from
$\tG^t$ by the time change $\DS \frac{ds}{dt}=(2 K_{n(t)})^{-1/2}$,
where $n(t)$ is the number of collisions before time $t$. Then the proof of 
Lemma~\ref{LmGB-LG} shows that $\hG^s$ is well approximated by $G^s$ for large values
of the kinetic energy. }

Accordingly we replace $\fM_\eps$ by the family $\fM_{\eps, t}$ consisting of the measures $\fm$
such that 

{(i) all standard pairs $\fm$ are longer than $\eps$ and;} 

(ii) $\fm$ is supported on the set 
$\{  x:\heps \leq  K(x)/t^{2/3}<1/\heps\} $ where $\heps$ is chosen so that
$$ \Prob\left(2\heps<\frac{ \widehat\cK(u)}{t^{2/3}}<\frac{1}{2\heps} 
\text{ for all } u\in [t/2, t]\right)\geq 1-\frac{\eps}{100},$$
where $\widehat\cK$ is the solution of \eqref{K-SDE-cont}.

Next we replace (M3) by 

$\widetilde{\text(M3)}$: For all $\fm\in \fM\; \forall  \tau \;\exists T: \forall t\geq T$
$$
\fm\left(x: \heps<\frac{K(x)}{t^{2/3}}<\frac{1}{\heps} \text{ but }
\sup_{s\in [0, \tau]} d(\tG^{s/{ \sqrt{2K(x)}}} x, G^s (x))>\eps\right)\leq \eps.
$$

and replace by (M5) by 

\smallskip
\noindent
$\widetilde{\text(M5)}$ For each $\fm\in\fM$ for each $\eps>0$ and 
$s { \geq} 0$ there exists
$T$
such that for $t\geq T$ we can decompose
$$ \tG^{{ t-s/t^{1/3}}} _* \fm={ \left[\sum_j c_j \fm_j \right] +c_{err} \fm_{err}}, $$ 
{ where for all $j$, $\fm_j\in \fM_{\eps,t}$ and there is some $\kappa_j$ such that  $\fm_j$ is supported on 
$\{|K(x)-\kappa_j|\leq 1\}$. Furthermore, ${ c_{err}}  \leq \eps.$}

The verification of (M1), (M2), $\widetilde{\text(M3)},$ (M4), $\widetilde{\text(M5)},$ (M6) is similar to
the verification of { (M1), (M2), $\overline{\text(M3)},$ (M4), 
(M5), $\overline{\text(M6)}$} for the collision map $\tT.$

{ Next, we explain what adjustments are needed in the proof of Theorem
\ref{ThLGM-App} (and its continuous time counterpart) to verify that 
$\widetilde{\text(M3)},$ $\widetilde{\text(M5)},$
can be used in lieu of (M3) and (M5) to infer local global mixing.

First, given $\Phi \in \bbG_U$, $\fm\in \fM$, $\delta>0,$ and $s>0$, we choose
$\eps >0$ small and apply $\widetilde{(M5)}$ 
to conclude that for all sufficiently large $t$
{
$$ \left| \fm\left(\Phi \circ \tG^t\right)-\sum_j  c_j \fm_j\left(\Phi \circ \tG^{s/t^{1/3}}\right)\right|\leq 
\delta .$$}
Further increasing $t$ if necessary, 
the bounded oscillation of $K(.)$ on $\fm_j \in \fM_{\eps,t}$ becomes negligible compared to $t$: specifically, for sufficiently large $t$, we have

$$ \left|\fm_j\left(\Phi \circ \tG^{s/t^{1/3}}\right)-
\fm_j\left(\Phi \circ \tG^{s\rho_j / \sqrt{2K(x)}} \right)\right|\leq \delta,$$
for all $j$, where $\DS \rho_j=\frac{\sqrt{2\kappa_j}}{t^{1/3}}$.
Next, by the definition of $\fM_{\eps, t}$, we have 
$\DS 2\sqrt{\heps} \leq \rho_j \leq 2/\sqrt{\heps}$. 
Thus we can use $\widetilde{(M3)}$ with  $\tau$ replaced by $2\tau/\sqrt{\heps}$ to conclude
that 
$$
 \left|
\fm_j\left(\Phi \circ \tG^{s\rho_j / \sqrt{2K(x)}} \right)
- 
\fm_j\left(\Phi \circ G^{s\rho_j} \right)
\right|\leq \delta
$$
Combining the last three displays, we get
\begin{equation}
\label{eq:lgGbflowpf}
\left| \fm\left(\Phi \circ \tG^t\right)-
\sum_j  c_j \fm_j\left(\Phi \circ G^{s\rho_j} \right)\right|\leq 
3\delta 
\end{equation}
As in the proof of Theorem \ref{ThLGM-App}, it is sufficient to verify that
$$\lim_{t \to \infty}\fm\left(\Phi \circ \tG^t\right) = \bar \Phi.$$ Thus by \eqref{eq:lgGbflowpf},
it suffices to verify that 
$$
|\fm_j\left(\Phi \circ G^{s\rho_j} \right) - \bar \Phi| < \delta
$$
for all $j$. This can be done by choosing $s = s(\delta)$ large and using the MLLT 
for $G$. This completes the proof of the local global mixing of $\tG$.}
\vskip2mm

{
{\bf Global global mixing for $\tG^t.$}
The proof is a simplified version of the proof of Theorem \ref{LLTIMMIx}(b)
because we have now $\Phi_1, \Phi_2 \in \bG_U$. Namely, we decompose
$$ \int_{\Omega_{\fq, R}} \Phi_1(x) \Phi_2(\tG^t x) d\mu(x)=
\sum_z \int_{\Omega_{\fq, R}} \Phi_1(x) 1_{z(x)=z} \Phi_2(\tG^t x) d\mu(x)
$$
where $z(x)$ is the label of the fundamental domain containing $x$.
We claim that if $R$ is sufficiently large, then there is a set $\brOmega\subset \Omega_{\fq, R}$ which is a union of fundamental
domains, such that 
$\DS \frac{\mu\left(\Omega_{\fq, R}\setminus \brOmega\right)}{\mu(\Omega_{\fq,R})}
=O(R^{-1/5}) $ and for $x\in \brOmega$, 
$\DS \min_{u\leq t} q_1(\tG^u x)\geq R^{1/10}.$ Indeed suppose that $R>t^{50}$ and
let $\brOmega$ be the union of fundamental domains such that $q_1(x)>R^{1/5}$ everywhere
on the domain. Using the fact that the speed of the particle is $O(R^{1/10})$ to the left
in the strip $0\leq q_1\leq R^{1/5}$, we conclude that for $x\in \brOmega$
$$ \min_{u\leq t} q_1(\tG^u x)\geq R^{1/5}-C R^{1/10} t\geq R^{0.2}-C R^{0.12}\geq 
R^{1/10} $$
for $R$ large, which proves the claim.

Arguing the same way as in the proof of local global mixing, we conclude that for 
the fundamental domains in $\brOmega$
$$ \int \Phi_1(x) 1_{z(x)=z} \Phi_2(\tG^t x) d\mu(x)=
\left[\int \Phi_1(x) 1_{z(x)=z}  d\mu(x) \right]\brPhi_2+o_{t\to\infty, R\to\infty}(1). $$
Since $\Phi_1\in \bbG_U$, we obtain
$$ \frac{1}{\mu(\Omega_{\fq, R})} \sum_z \int_{\Omega_{\fq, R}}  \Phi_1(x) 1_{z(x)=z}  d\mu(x)$$
$$=
\frac{1}{\mu(\Omega_{\fq, R})} \int_{\Omega_{\fq, R}} \Phi_1(x)  d\mu(x)=
\brPhi_1+o_{R\to\infty}(1) $$
completing the proof of global-global mixing.}
\end{proof}

\subsection{Fermi-Ulam pingpong.}
\label{SSFUPp}
Consider the following one-dimensional system: a unit point mass moves horizontally between two infinite mass walls. Between collisions, the motion is free so that the 
kinetic energy is conserved, collisions between the particle and the
walls are elastic. The left wall moves periodically, while the right one is fixed. The distance between the two walls
at time $t$ is denoted by $\ell(t)$. We assume that $\ell$ is strictly positive, 
continuous and periodic of period $1$. Moreover we suppose that the restriction of $\ell$ to the open interval $(0,1)$
is $C^5$ but $\dot \ell(1-) \neq \dot \ell (1+)$, where
$\dot \ell(1+) = \lim_{t\downarrow 0} \dot\ell(t)$ and $\dot \ell(1-) = \lim_{t\uparrow 0} \dot\ell(t).$
Thus $\ell$ is piecewise smooth with singularities only at integers.  
Let $\tT$ be the map defined as follows. Let the particle move until the the next integer moment of time 
and then stop it after the first collision with the moving wall. Note that $\tT$ is conjugated to 
$G$-the time $1$
map of the system. Namely for $\tT$ it is natural to use the following coordinates:
the time of collision (taken modulo $\integers$) and the post collisional velocity at the
moment of collision. For $G$ it is natural to use velocity and height. To pass from the first coordinate
set to the second one, we replace the post collisional velocity with the precollisional one and then let the
particle move backward until the first time it becomes an integer. 

It is shown in \cite{dSD12} that $\tT$ is well approximated at infinity by the following map of
the cylinder $\Tor\times \reals :$
\begin{equation}
\label{PPLim}
 T(\tau, I)=( \tau-I, I+\Delta(\tau-I))
\end{equation}
where  
$$ \Delta=\ell(0) \sigma \int_0^1 \ell^{-2}(s)\; ds, \quad
\sigma=
\dot \ell(1+) - \dot \ell (1-).
$$
$T$ covers a map $f$ of $\Tor^2$ which is defined by formula \eqref{PPLim} with 
$I$ taken mod 1.
{Specifically, property (ii) of Definition
\ref{DefWApp2} holds with $d_1 = 1$, $d_2 = 0$.}
If $\Delta\not\in (0,4)$ then the map $f$ is piecewise hyperbolic and according to 
\cite[Section 7]{Y98}, it admits a Young tower and hence, satisfies the MLLT
(see e.g. \cite{G05}). Therefore in this case $\tT$ and, hence, $G$ are global global mixing with respect
to $\bbG_U.$
 
We note that while the dynamics for large energies is described by a single parameter $\Delta$, the dynamics
for low energies is far from universal. In particular, it is easy to construct an example where
$T$ has elliptic fixed points and so it is not ergodic. Thus we get another natural example where
the map is global global mixing but is not ergodic. 

On the other hand it is shown in \cite{dSD18} that if $\ell$ is piecewise convex, then $\tT$ is
ergodic for most values of the parameter $\Delta$ (with at most a countable set of exceptions).
One could expect that in that case $\tT$ is local global mixing, but this question requires a further 
investigation. 

\subsection{Bouncing ball in a gravity field.}
\label{SSGravity}
In this model a particle moves on $\mathbb R_+$ in a linear potential $U(x)=gx$ and collides 
elastically with an infinitely heavy wall whose position at time $t$ equals to $h(t)$. We assume that $h$
is 1-periodic and piecewise $C^2$ but not $C^2$. Let $\tT$ be the collision map in this model. It is shown in \cite{Z18}
that $\tT$ is well approximated at infinity by the map $T$ of the cylinder $\Tor\times \reals$ given by
\begin{equation}
\label{PPGLim}
T(t,v)=(t+2v/g, v+2 \dot{h}(t+2v/g) ).
\end{equation}
 $T$ is a $\integers$ cover of the map $f$ of $\Tor^2$ defined by \eqref{PPGLim} with
$t$ taken mod 1 and $v$ taken mod $\frac{g}{2}.$
{(Again, property (ii) of Definition
\ref{DefWApp2} holds with $d_1 = 1$, $d_2 = 0$.)}
Moreover, it is proven in \cite{Z18} that if either
\begin{equation}
\label{JingCond}
 \ddot{h}>0 \text{ or } |\ddot{h}+a|\leq \eps 
\end{equation} 
where $a>g$ and $\eps=\eps(a)$ is a small constant, then $f$ satisfies the conditions of \cite{CWZ}.
Consequently it admits a Young tower with exponential tail and hence satisfies the MLLT.
It follows
{from Theorem \ref{ThInfAppGGM}}
that if \eqref{JingCond} is satisfied, then $\tT$ enjoys global global mixing with respect to 
$\bbG_U.$

As in the previous example, the dynamics for small energies is not universal
and the question about local global mixing may depend on the law energy dynamics of the system.
Finally we note that the continuous time system is not global global mixing since 
on most of the phase space the motion is integrable. Namely let $\Phi$ be a non negative continuous
function which depends only on velocity, is 1-periodic and is supported on 
$\{v: d(v, \integers)\leq 0.01\}.$ Then $\brPhi=\int_0^1 \Phi(v) dv>0.$ On the other hand for each
$T,$ on most of the set $\{ v \leq V\}$ 
with $V \gg T$, velocity remains large on the time interval $[0, T].$ 
For such orbits $v(t)=v(0)-gt$ for $t\in [0, T]$ and so if 
$d(gT, \integers)>0.04$ then $\Phi\cdot (\Phi\circ \tG^T)=0.$ Accordingly the large volume limit for such $T$'s is
$$\overline{\Phi\cdot (\Phi\circ \tG^T)}=0$$
precluding global global mixing.  As in the discrete time case the question of local global mixing is more
subtle and deserves a further investigation.

\section{{Condition (M6) for Lorentz gas with external fields}}
\label{sec:m6}

Here we {complete the proof of Theorem
\ref{thm:fieldlg} by checking} the condition (M6) for Lorentz gas with vanishing potential.
We hope that similar arguments will apply to other hyperbolic systems with singularities,
  including the examples of \S \ref{SSFUPp} and \S \ref{SSGravity} once their dynamics in
  the low energy regime is better understood. 

\subsection{Recurrence-transience dichotomy.}
For sets $\cA, \cB$ we shall write $\cA\equiv \cB$ if their symmetric difference satisfies
$\mu(\cA\triangle \cB)=0.$

\noindent
In this section we prove an auxiliary result of independent interest.
Let
$$\cR^\pm=\{x: |z(\tT^n x)|\not\to\infty \text{ as } n\to\pm\infty\}. $$
Then, (see e.g. \cite[\S 1.1]{A97}), $\cR^-\equiv \cR^+.$ Let
$\cR=\cR^-\cap \cR^+$ be the set of recurrent orbits. Then $\cR\equiv \cR^+\equiv \cR^-.$

\begin{lemma}
Either ${\mu}(\cR)=0$ or ${\mu}(\cR^c)=0.$ In the second case, $\tT$ is ergodic. 
\end{lemma}  

\begin{proof}
Let $\cR_0=\cR,$ $\cR_0^\pm=\cR^\pm,$
and for $n>0$ define inductively $\cR_n=\cR_n^+\cap \cR_n^-$ where
\begin{gather*}\cR_n^+=\{x\in\cR_{n-1}: \mes(W^s(x)\cap \cR_{n-1}^c)=0\},\\
\cR_n^-=\{x\in \cR_{n-1}: \mes(W^u(x)\cap \cR_{n-1}^c)=0\}.
\end{gather*}
We shall show inductively that 
\begin{equation}
\label{AllREquiv}
\cR_n\equiv \cR_n^+\equiv \cR_n^-=\cR_{n-1}.
\end{equation}
For $n=0$ this follows from the foregoing discussion. Assuming that
\eqref{AllREquiv} holds for $n-1$ we obtain, using the  absolute continuity of the stable  lamination (namely, \eqref{AC2})
and the relation $\cR_{n-1}\equiv \cR_{n-1}^+$, that
$$ \cR_n^+\equiv \{x\in \cR_{n-1}^+: \mes(W^s(x)\cap (\cR_{n-1}^+)^c)=0\}\equiv \cR_{n-1}^+ $$
where the last step uses that, by construction, 
$$\mes(W^s(x)\cap (\cR_{n-1}^+)^c)=0$$ for $x\in \cR_{n-1}^+.$
Thus $\cR_n^+\equiv \cR_{n-1}.$
Likewise $\cR_n^-\equiv \cR_{n-1},$ proving \eqref{AllREquiv}.
\eqref{AllREquiv} shows that
\begin{equation}
\label{RInfty-R}
\cR_\infty:=\bigcap_n \cR_n\equiv \cR.
\end{equation}

Let $\cE_0=\cE=\cE^+\cap \cE^-$ where
$$ \cE^\pm=\{x: |z(\tT^n x)|\to\infty \text{ as } n\to\pm\infty\}.$$
and define $\cE_n$ and $\cE_\infty$
similarly to $\cR_n$ and $\cR_\infty$ respectively.
Similarly to \eqref{RInfty-R} we obtain that
$$ \cE_\infty\equiv \cE\equiv\cE^+\equiv \cE^-. $$
Denote $\cG=\cE_\infty\cup\cR_\infty.$ By the foregoing discussion 
$$\cG\equiv \cE\cup\cR\equiv \cE^+\cup \cR^+. $$
Since the last set equals to the whole phase space we conclude that
 $\mu(\cG^c)=0.$ 
 
Suppose for a moment that that $\cR_\infty\neq \emptyset.$
Pick $x'\in \cR_\infty$. {Then, by \cite[Lemma 3.6]{C08}} for
every $x''\in \cG$ there exists a Hopf chain, that is, a chain
$$x'=y_0, y_1,\dots, y_n=x'' \text{ such that } y_j\in \cG\text{ and }y_{j+1}\in
W^s(y_j) \cup W^u(y_{j}). $$
By construction since $y_0=x'\in\cR_\infty$ then $y_j\in \cR_\infty$ for all $j.$
Thus $x''\in \cR_\infty$ and hence $\mu(\cR^c)=0.$ 

On the other hand
if $\cR_\infty=\emptyset$ then $\mu(\cR)=0.$ This proves the first claim
of the lemma. The fact that recurrence implies ergodicity follows from
\cite{L03}.
\end{proof}

\begin{corollary}
\label{CrNonLocalize}
For any set $A$ of finite measure and for any $\eps, R>0$ there exists $n$
such that
\begin{equation}
\label{SomeTimeFar}
{\mu}(x\in A: \tT^n x\in B_R)< \eps,
\end{equation}
where $B_R=\{x: |z(x)|\leq R\}.$
\end{corollary}

\begin{proof}
  If ${\mu}(\cR)=0$ then $\tT$ is dissipative (\cite[\S 1.1]{A97}), that is, for a.e. $x$
  $$ \lim_{n\to+\infty} |z(\tT^n x)|=+\infty,$$ so
  \eqref{SomeTimeFar} is obvious.

  On the other hand if ${\mu}(\cR^c)=0$ then $\tT$ is ergodic, so the
  Ratio Ergodic Theorem tells us that for each
  $z_1, z_2$ and
  for almost every $x$
    $$ \lim_{N\to\infty}
  \frac{\Card(n\leq N: z(\tT^n x)=z_1)}{\Card(n\leq N: z(\tT^n x)=z_2)}=
  \frac{ {\mu}(x: z(x)=z_1)}{{\mu}(x: z(x)=z_2)}. $$
  Since the last expression is uniformly bounded away from 0 we have
  that for any $\brz$ and almost every $x$
  $$ \lim_{N\to\infty}
  \frac{\Card(n\leq N: z(\tT^n x)=\brz)}{N}=0. $$
  By the Dominated Convergence Theorem
  \begin{gather*}
  \frac{1}{N} \sum_{n=1}^N {\mu}(x\in A: z(\tT^n x)=\brz)\\
  =
  {\mu}\left(\frac{\Card(n\leq N: z(\tT^n x)=\brz)}{N} 1_{\{x \in A\}}\right)\to 0
  \text{ as $N \to \infty$.}
  \end{gather*}
Summing over $\brz$'s such that $|\brz|\leq R$ we get
$$ \frac{1}{N} \sum_{n=1}^N  \tilde \mu(x\in A: \tT^n x\in B_R)\to 0 .$$
Therefore the set of times $n$ when
\eqref{SomeTimeFar} is false has zero density.
\end{proof}  

The preliminaries discussed in Section \ref{BilPrel} extend to the case of billiards will small external fields by
\cite{Ch01, C08}. In particular for an unstable curve $\gamma$, we write 
$$\gamma_\delta=\{x\in \gamma: r_s(x)\geq \delta\},\quad  \Lambda_\delta(\gamma)=\bigcup_{x\in \gamma_\delta} W^s(x). $$
Then \eqref{GL-ShMan} holds 
{ (see \cite[Lemma 3.2]{C08} in case of external fields)}
and we have the analogue of \eqref{AC3}:
\begin{equation}
  \label{AC3v2}
\kappa_1\leq \frac{d\hmu}{d{\mu}_{\Lambda_\delta}}\leq \kappa_1^{-1}. 
\end{equation}
and the analogue of \eqref{AC4}:
\begin{equation}
  \label{AC4v2}
{\mu}(\Lambda_\delta(\gamma))\geq \kappa_2. 
\end{equation}

\begin{corollary}
\label{CrNonLocalizeLoc}
For any unstable curve $\gamma$ for any $\eps, R>0$ there exists $n$
such that
\begin{equation}
\label{SomeTimeFarLoc}
\mes(x\in \gamma: \tT^n x\in B_R)< \eps.
\end{equation}
\end{corollary}

\begin{proof}
Since measure of $\gamma-\gamma_\delta$
tends to 0 as $\delta\to 0$ (see \eqref{GL-ShMan}), it suffices to prove that, for each fixed $\delta,$ 
\eqref{SomeTimeFarLoc} holds with $\gamma$ replaced by $\gamma_\delta.$
Combining Corollary \ref{CrNonLocalize} with \eqref{AC3v2} we obtain 
for each $\eps>0$ 
there exists $n$ such that
$$ \hmu(x\in \Lambda_\delta: |z(\tT^n x)|\leq R+1)< \eps. $$
On the other hand the definition of $\hmu$ easily shows that
$$ \hmu(x\in \Lambda_\delta: |z(\tT^n x)|\leq R+1) \geq  \delta\mes(x\in \gamma_\delta: 
|z(\tT^n x)|\leq R)$$ 
proving the result.
\end{proof}

\subsection{Verifying (M6)} 
\label{SSM6}
By our choice of $\fM$ it suffices to show that for each $\delta,$
 for each $\eps$ and $R$ there exists $n_0$ such that
for $n\geq n_0$ for each { unstable}
curve $\Gamma$ of length at least $\delta$ we have 
\begin{equation}
\label{LeaveLoc}
\mes(x\in\Gamma: \tT^n x\in B_R)\leq \eps. 
\end{equation}
We first show this result under an additional assumption that 
\begin{equation}
\label{ZSuperLarge}
|z(\Gamma)|\geq \tR 
\end{equation}
provided $\tR=\tR(\eps, \delta, R)$ is sufficiently large
and then use Corollary \ref{CrNonLocalizeLoc} to remove this restriction.

Before giving the formal proof let us describe the main idea. Given an unstable curve $\Gamma$
satisfying the conditions above and $\tn\in\naturals$ we consider the Hopf $\tn$-brush obtained by
issuing the stable manifolds from all points of $\tT^\tn \Gamma.$ We shall show that 

(i) If $\tn= \tn(\eps, \delta, R)$ is large, then the brush has a large measure;

(ii) If at some time $n\geq \tn$ a significant proportion of $\Gamma$ came close to the origin,
then a significant portion of the $\tn$-brush would come close to  the origin at time $n-\tn.$
Since $\tT^{n-\tn}$ is measure preserving, there is not enough room in a fixed neighborhood
of the origin, giving a contradiction.

To prove part (i) above we show that the image $\tT^\tn\Gamma$ stretches across a large number 
of cells. For $T$ this is true because of the LLT, while for $\tT$ this is true because it is 
{very} well
approximated by $T$ at infinity (at this step it is important that we take 
$\tR= \tR( \eps, \delta, R, \tn)$ sufficiently large). Next, the Growth Lemma implies that
most of the components of $\tT^\tn \Gamma $ are not too short. Consequently, there are many cells
whose intersection with $\tT^\tn \Gamma$ contains relatively long component. 
Now \eqref{AC4v2} implies that the brush has a
significant measure in each such cell.


The proof of part (ii) uses the fact that if a point
returns close to the origin then the same is true for its whole (homogeneous) stable manifold.

We now give a more detailed argument. 
{ We divide the proof into 
seven steps.} \medskip

{\bf Step 1: Preliminaries.}

Let $\delta_1\ll \delta$ be a small constant. The precise requirements on $\delta_1$ will be given below.
Here we require that for each {unstable} curve $\Gamma$ of length at least $\delta$ and for each $n$,
\begin{equation}
\label{NotGood}
\mes(x\in \Gamma: x\text{ is not }(\delta_1, n)-\text{good})\leq \eps^2,
\end{equation}
where we call $x$ {\em $(\delta_1,n)$-good} if
\begin{equation}
\label{eq:notgoodspellout}
r_n(x)\geq \sqrt{\delta_1} \text{ and } r_s(\tT^n x)\geq \sqrt{\delta_1}. 
\end{equation}
(The existence of $\delta_1$ { when only the first inequality is required
in \eqref{eq:notgoodspellout}
follows from the 
Growth Lemma \ref{LmGrowth} (\cite[Proposition 5.3]{Ch01} in case of external fields). The second inequality
can also be ensured by
combining \eqref{GL-ShMan} (\cite[Lemma 3.2]{C08} in case of external fields) with (M1)).}

By transversality of stable and unstable directions, there is a constant
$K_1$ such that if $\cT$ is an unstable curve and $\pi$ is the projection to
$\cT$ along the stable leaves, then 
\begin{equation}
\label{ProjLip}
d(\pi x, x)\leq K_1 d(x, \cT)
\end{equation}
 provided that $\pi$
is defined at $x$.
\medskip

{
{\bf Step 2: Long brushes are abundant.}}
Let
$$X_{\tk, \eta}=\{x\in X: \forall y\in B(x, \eta)\; \forall\; 0\leq j\leq \tk\;\; \tT
\text{ is continuous on } B(\tT^j y, \eta)\}, $$
and define $M_{\tk, \eta}$ similarly with {$X$ replaced by $M$ and}
$\tT$ replaced by $T$.
{In step 2, we prove that for $\tk$ large enough and for $\delta_1 = \delta_1(\tk)$
sufficiently small the following holds.}
If $x\in X_{\tk, 2K_1\delta_1}$ 
and $\cT$ is an unstable curve of length
$\delta_1$ through $x$, then
\begin{equation}
\label{ManyLongWs}
\mes(\ft'\in\cT: r_s(\ft')\geq 2K_1 \delta_1)\geq \frac{\delta_1}{2}.
\end{equation}

{To prove \eqref{ManyLongWs}, first we recall inequality (5.58) from \cite{CM06}}: 
$$ r_s(\ft')\geq \min_{n\geq 0} \Lambda^n {d^s}(\tT^n \ft', \cS)$$
where $\Lambda > 1$ is the minimal expansion factor of
$\tT$, $\cS$ is the discontinuity set of $\tT$
{and $d^s(\tT^n \ft', \cS)$ is the length of the shortest unstable curve
that connects  $\tT^n \ft'$ with the set $\cS$.}

Note that if the above minimum falls below $2K_1 \delta_1$, then also
\begin{equation}
  \label{SingLate}
  \min_{n\geq \tk} \Lambda^n  {d^s}(\tT^n \ft', \cS)\leq 2K_1\delta_1.
\end{equation}  
{(Indeed, for $n < \tk$, 
$$\Lambda^n d^s(\tT^n \ft', \cS) \geq d^s(\tT^n \ft', \cS) \geq d(\tT^n \ft', \cS) 
\geq 2K_1\delta_1$$
by the definition of $X_{\tk, 2K_1 \delta_1}$.)
Let us write $\ell = (\cT, \frac{1}{\delta_1}\mes_{\cT})$.
Then, we have
\begin{align*}
&\nu_{\ell} (\ft'\in \cT:
\min_{n\geq \tk} \Lambda^n  d^s(\tT^n \ft', \cS)\leq 2K_1\delta_1) \\
&\leq \sum_{n=\tk}^{\infty} 
\nu_{\ell}
(\ft'\in \cT:  d^s(\tT^n \ft', \cS)\leq \Lambda^{-n} 2K_1\delta_1).
\end{align*}
Next, observe that by transversality there exists some constant $C$ so that for every 
$\ft \in \cT$, $r_n(\ft') \leq C d^s(\tT^n \ft', \cS)$. Thus the above display can be bounded by
$$\sum_{n=\tk}^{\infty} 
\nu_{\ell}
(\ft'\in \cT:  r_n(\ft') \leq \Lambda^{-n} 2CK_1\delta_1) \leq
\sum_{n=\tk}^{\infty}
 \cZ(\tT^n_* \ell) \Lambda^{-n}2CK_1\delta_1
$$
Using the fact that
$\cZ_{\ell} = 2/\delta_1$ and the growth lemma, the above is bounded by
$$
\sum_{n=\tk}^{\infty}(C_1 \theta^n \frac{2}{\delta_1} + C_2)( \Lambda^{-n}2CK_1\delta_1)
= \frac{4K_1CC_1}{1-\theta/\Lambda} \theta^{\tk} \Lambda^{-\tk} +
\frac{2K_1CC_1 \delta_1}{1-1/\Lambda} \Lambda^{-\tk} =: I + \RmII.
$$
Now we choose $\tk$ so that $I < 1/4$ and then choose $\delta_1 = \delta_1(\tk)$ so that
$\RmII < 1/4$. Since $\nu_{\ell} = \frac{1}{\delta_1} \mes_{\cT}$, 
\eqref{ManyLongWs} follows.
}

{To complete Step 2 we show that 
$X_{\tk, 2K_1\delta_1}$ fills most of the space.
Namely, by further reducing $\delta_1 = \delta_1(\tk)$ if necessary, we may assume that
\begin{equation}
\label{eq:cont.square}
\mu(M-M_{\tk, 2K_1\delta_1})\leq  \eps^{7}.
\end{equation}
Then for large $\tR$ and for each cell $\cC = \{ z = m \}$ which is at least $\tR$ away from the origin, 
\begin{equation}
\label{WShortSmall}
{\mu} ((X - X_{\tk, 2K_1 \delta_1}) \cap \cC) <  2\eps^{7}.
\end{equation}
\medskip

{\bf Step 3: Construction of unstable frame.}
Next, we construct a collection of unstable curves 
$\{ W_{\bk ,i,j} \}$, $i=1,..,I$, $j=1,...,J$, $\bk \in \integers^2$ with $W_{\bk,i,j}
\subset X \cap \{ z = \bk \}$ with $\mbox{length}(W_{\bk ,i,j}) \in [\delta_1, 2 \delta_1)$
that will serve as the handles  of our brushes.

Recall that by \eqref{eq:uniformcones},
the unstable cones can be defined in a way that there is a segment
$[\alpha, \gamma] \subset \cS^1$ (here $\cS^1$ is identified with 
$ [0, 2 \pi)$)
so that $0 <\alpha < \gamma < \pi/2$
and for any $y \in M$ and for any $\beta \in [\alpha, \gamma]$, the direction
$\beta = d \phi / dr$ is in the unstable cone. Increasing $\alpha$ and decreasing
$\beta$ a little and  supposing that the field small enough, the same is true
for $(y,\bk) \in X$ for any $y \in M$ and $\bk \in \integers^2$.
Let us now fix $\bk \in \integers^2$. 
First we fix parallel lines 
$\cW_1,...\cW_I \subset X \cap\{z = \bk\}$
with angle
$d \phi / dr = \beta$ where $\beta := (\alpha + \gamma)/2$
and the distance between $\cW_i$ and $\cW_{i+1}$ is $\delta_1$. 
(To be more precise, we have to fix these lines in all connected components of 
$X \cap \{ z = \bk\}$,
which are
topological cylinders, but to simplify notation we pretend that there is only 
one cylinder. 
Also we do not emphasize the dependence on $\bk$ as the curves
are translates of one another for different $\bk$'s). 
Each line segment $\cW_i$ connects the two
boundaries of the cylinder, that is one of its endpoints 
is on the line $\phi = - \pi/2$, the other one is on the line 
$\phi = \pi/2$. The index $I$ is defined by
$$
I = \max \{ i: i \cos (\beta) \delta_1 \leq \text{ arc length of the
scatterer} \} -1.
$$


We would like to use $\cW_i$'s as the frame for building our brushes,
However, there are two problems
 when trying to use \eqref{ManyLongWs}. First, 
$\cW_i$'s are too long
 compared to $\delta_1$,  so the right hand side of 
 \eqref{ManyLongWs} does not give a good bound for the relative measure on $\cW_i$. 
Secondly, $\cW_i$ may be disjoint to $X_{\tk, 2K_1\delta_1}$
and so \eqref{ManyLongWs} may not hold.
To handle the first issue we subdivide each $\cW_i$ into shorter pieces. 
To handle the second issue we perturb slightly each short segment so that the resulting broken line 
lies in a $\xi \delta_1$ neighborhood of $\cW_i$
 and most of the resulting segments 
$\{W_{\bk,i,j}\}_{j=1,\dots, J}$ contain a point in 
$X_{\tk, 2K_1\delta_1}$. They are defined as follows.
$W_{\bk,i,j}$ is the line segment connecting $(r_{\bk, i,j-1}, \phi_{\bk,  i,j-1},\bk)$
and $(r_{\bk,i,j}, \phi_{\bk,i,j},\bk)$,
where $\phi_{\bk,i,j} = -\pi/2 + j \sin (\beta) \delta_1$ 
 for 
$$j < J := \max \{ j: j \sin (\beta) \delta_1 < \pi\}$$
and $\phi_{\bk,i,J} = \pi/2$,
and $r_{\bk,i,j}$ is defined
inductively. First, $r_{\bk,i,0}$ is 
such that $(r_{\bk,i,0}, - \pi/2)$ is an endpoint of $\cW_i$
and denote $\hat r_{\bk,i,j} = r_{\bk,i,0} + j \cos (\beta) \delta_1$ (thus
$(\hat r_{\bk,i,j}, \phi_{\bk,i,j},\bk) \in \cW_i$).
 Now assume that 
$r_{\bk,i,j}$ is defined so that $r_{\bk,i,j} - \hat r_{\bk,i,j} \in (-\xi \delta_1, \xi \delta_1)$.
If $r_{\bk,i,j} - \hat r_{\bk,i,j} <0$ ($>0$, resp.), then we try to choose
$r_{\bk,i,j+1} \in 
(\hat r_{\bk,i,j+1}, \hat r_{\bk,i,j+1}+ \xi \delta_1)$ 
(respectively $r_{\bk,i,j+1} \in 
(\hat r_{\bk,i,j+1} - \xi \delta_1, \hat r_{\bk,i,j+1})$)
so that
the line segment $W_{\bk,i,j}$ contains a point in $X_{\tk, 2K_1\delta_1}$. If this
is not possible, we choose $r_{\bk,i,j+1}$ arbitrarily (in the above interval) and 
say that $W_{\bk,i,j}$ is bad. Note that in case $W_{\bk,i,j}$ is bad, then there is a corresponding bad region
of area $C\delta_1^2$ that is disjoint to $X_{\tk, 2K_1\delta_1}$. 

To facilitate the
comparison between the invariant measure $\mu$ and the area, we say that 
$W_{\bk,i,j}$ is marginal if
$\min \{j, J-j\} < \eps^2 / (2\delta_1)$. 
Thus there are three kinds of line segments $W_{\bk,i,j}$: marginal, bad 
(from now on {\em bad} means bad in the sense defined above, but
not marginal) and good.

Now if $W_{\bk,i,j}$ is bad,
then the $\mu$ measure of the corresponding bad region is at least 
$C \eps^4 \delta_1^2$ and so 
by \eqref{WShortSmall},
the number of bad curves for any $\bk$
is bounded by $ \eps^2 \delta_1^{-2}/2$.
Also, the $\mu$ measure of the $K_1 \delta_1$ neighborhood of marginal curves
is bounded by $ \eps^2/2$.}
\medskip

{
{\bf Step 4: Anticoncentration of measure.}}
Next, pick an unstable curve $\Gamma$ of length at least $\delta$ satisfying \eqref{ZSuperLarge}.
Let 
{$\cT$ be the union of 
the line segments $\{ W_{k,i,j}\}$ constructed in Step 3.} Given $\tn\in\naturals$ 
let $\pi_\tn: \tT^{\tilde n} \Gamma\to\cT$ be the projection to the closest {$ W_{\bk,i,j}$}
along the stable leaves.
Assuming that $\delta_1$ is so small  that $\sqrt{\delta_1} > K_1 \delta_1$
we get that $\pi_\tn$ is defined on $\tT^\tn x$ if $x$ is 
$(\delta_1,  \tn)$-good.
Denote by $J_\tn$ the {inverse of the} Jacobian of $\tT^\tn:\Gamma\to\tT^\tn\Gamma.$
For $\ft\in \cT$ let 
$$ \cJ(\ft)=\sum_{\stackover{x\text{ is } (\delta_1, \tn)-good}
{\pi_\tn (\tT^\tn x)=\ft}} J_{\tilde n}(x). $$

Let $L_\tn=\{\ft\in \cT: 0<\cJ_\tn(\ft)<\frac{1}{\sqrt{\tn}}\}. $
{In Step 4, we prove the following claim: if $\tn = \tn(\delta_1), \tR = \tR(\delta_1, \tn)$ are large enough, 
$W_{\bk, i,j}$ is a good line segment constructed in Step 3,
$\ft \in W_{\bk, i,j}$ and $\cJ(\ft) >0$, then $\ft \in L_\tn$.}

To prove this claim, first we  
observe that by the definition of
$\pi_\tn$ and \eqref{ProjLip},
if $\pi_\tn (\tT^\tn x)=\ft$, then $d(\tT^\tn x, \ft)\leq K_1\delta_1.$
Take $\ft'$, on the same {$ W_{\bk,i,j}$} as $\ft$ with
$r_s(\ft')\geq 2 K_1 \delta_1$ 
({the Lebesgue measure of such points is at least $\delta_1/2$ by \eqref{ManyLongWs}
by the fact that $W_{\bk,i,j}$ is good}).
Since $x$ is {$(\delta_1,\tn)$-good
and by the construction of $\cT$}, there is $x'\in \Gamma$ such that
$\tT^\tn x'$ belongs to the same component as $\tT^\tn x$ and
$\pi(\tT^\tn x')=\ft'$.
By bounded distortion of $\tT^\tn$ (see \eqref{BDWu}), there exists a constant $c$ such that if
$\cJ_\tn(\ft)\geq \frac{1}{\sqrt{\tn}}$, then
$\cJ_\tn(\ft')\geq \frac{c}{\sqrt{\tn}}.$
Combining the absolute continuity of $\pi_\tn$ (see
\eqref{HolJac} and \eqref{HoloBnd}) with \eqref{ManyLongWs} { 
(and noting that the length of $W_{\bk, i, j}$
is bounded by $2 \delta_1$ by construction),}
we conclude that if there existed $\ft'$ such that $\cJ_\tn(\ft')\geq \frac{1}{\sqrt{\tn}}$, then we would have
\begin{equation}
  \label{HitSquare}
 \mes(x\in \Gamma:  z(\tT^\tn x)=z(\ft)) \geq \frac{\bar{c} \delta_1}{\sqrt{\tn}} . 
\end{equation}  

On the other hand the LLT for $T$ shows that  there is a constant $ \tC$ such that for each $\tn$
there exists $\tR$ such that if $z(\Gamma)\geq \tR$, then
\begin{equation}
  \label{LLTBound}
\mes(x\in \Gamma:  z(\tT^\tn x)=z(\ft)) \leq \frac{\tC}{\tn}.
\end{equation}  
If $\tn$ is so large that
$ \DS \frac{\tC}{\tn}< \frac{\brc \delta_1}{\sqrt{\tn}},$ that is,
\begin{equation}
\label{TldN1}
\tn>\left(\frac{\tC}{\brc \delta_1}\right)^2,
\end{equation}
this
gives a contradiction with \eqref{HitSquare} proving the claim.
\medskip

{
 \bf Step 5: Most of the image of $\Gamma$ is not too close to the discontinuities.} 
We claim that
if $\delta_1$ is small, then for appropriate $\tn, \tR$ we have
\begin{equation}
\label{GammaStar}
 \mes(\Gamma\setminus \Gamma^*)\leq 4 \eps^2,
 \end{equation}
where
$ \Gamma^*$ is the set of points $x$ in $\Gamma$ such that 
$x$ is $(\delta_1, \tn)$--good and $\pi_\tn(\tT^\tn x)\in L_\tn.$

{ To
prove \eqref{GammaStar} note that by combining
 \eqref{NotGood} 
with the fact that for $(\delta_1, \tn)$--good points $x$,
$\pi_\tn(\tT^\tn x)$ exists, \eqref{GammaStar} will be implied by
the following:
$$
 \mes( \Gamma^{\#})\leq 3 \eps^2,
$$
where $ \Gamma^{\#}$ is the set of points $x$
in $\Gamma$ that are 
$(\delta_1, \tn)$--good
and 
$\pi_\tn(\tT^\tn x)\notin L_\tn.$ By Step 4, it
is sufficient to prove that the Lebesgue measure of points
$ x \in \Gamma$ so that $x$ is $(\delta_1, \tn)$--good
and $\pi_\tn(\tT^\tn x) \in W_{\bk,i,j} \in \cT$ with some marginal or
bad $W_{\bk,i,j}$ is bounded by $3 \eps^2$. 

Note that by choosing $\tR$ large we can ensure that
the goodness of $W_{\bk,i,j}$ only depends on $i,j$ and not on $\bk$
as long as $|\bk| > \tR - \tn $. Indeed, for fixed $\tk, \delta_1, \tn$ we can ensure that
the singularities of $\tT^{\tk + \tn}$ are uniformly close to those of 
$T^{\tk + \tn}$ by choosing the field small. 
Let us write $(i,j) \in \cB$ if $W_{\bk,i,j}$ is bad or marginal for some
(and hence for all) $\bk$ with $|\bk| > \tR - \tn$.

Next, increasing $\tn = \tn(\delta_1)$ if necessary, uniform equidistribution 
of the images of 
unstable curves
(see \cite[Proposition 2.2]{C08}) implies that
\begin{align*}
&\mbox{mes} (x \in \Gamma: \exists \bk, \exists (i,j) \in \cB:
\pi_\tn(\tT^\tn x) \in
W_{\bk,i,j} )\\
& \leq 2 
\mu ( x \in X: 
d(x, \cup_{(i,j) \in \cB} W_{\tilde{\bk},i,j}) < K_1 \delta_1 )
\end{align*}
where $\tilde{\bk}$ is arbitrary with $|\tilde{\bk}| > \tR$. The last displayed formula is bounded by $3 \eps^2$ by the last paragraph of Step 3. 
We have verified \eqref{GammaStar}.
}
\medskip

{
{\bf Step 6: Proof of 
\eqref{LeaveLoc} assuming \eqref{ZSuperLarge}.}}
By the definition of $L_\tn$, {for any $N > \tn$,}
\begin{equation}
\label{ReturnGammaStar}
\mes(x\in\Gamma^*: T^N x\in B_R)\leq 
\frac{1}{\sqrt{\tn}} \mes(y\in L_\tn: T^{N-\tn} y\in B_{R+1}) .
\end{equation}
On the other hand combining the absolute continuity of the stable lamination 
(see \eqref{AC3v2})
with the fact that $r_s\geq \delta_1$ on $L_\tn$, we obtain that 
there is a constant $\hC$ such that
\begin{equation}
\label{L-HL}
\mes(y\in L_\tn: T^{N-\tn} y\in B_{R+1}) 
\leq \frac{\hC}{\delta_1 }  {\mu}(y\in \hL_\tn: T^{N-\tn} y\in B_{R+2}), 
\end{equation}
where $\DS \hL_\tn=\bigcup_{z\in L_\tn} W^s(z).$ 

Since $\tT$ preserves ${\mu}$, we have
\begin{equation}
\label{Area}
 {\mu}(y\in \hL_\tn: \tT^{N-\tn} y\in B_{R+2})\leq D (R+2)^2
\end{equation}
for some $D>0.$ Combining \eqref{ReturnGammaStar}, \eqref{L-HL}, and \eqref{Area},
we see that
$$ \mes(x\in\Gamma^*: T^N x\in B_R)\leq \frac{D\hC(R+2)^2}{\delta_1\sqrt{\tn}}. $$
Thus if
\begin{equation}
\label{TldN2}
\tn\geq  \left[ \frac{D \hC (R+2)^2}{2 \delta_1 (\eps-4\eps^2)} \right]^2,
\end{equation}
then 
$$ \mes(x\in\Gamma^*: T^N x\in B_R)\leq \eps-4 \eps^2. $$

Combining this with \eqref{GammaStar} we obtain
\eqref{LeaveLoc} provided $|z(\Gamma)|$ is large as required by \eqref{ZSuperLarge}. \smallskip

\medskip

{
{\bf Step 7: Relaxing \eqref{ZSuperLarge}.}}
It remains to obtain \eqref{LeaveLoc} without assuming \eqref{ZSuperLarge}. Fix $\eps>0.$
Then take $\delta_2$ so small that for every unstable curve $\Gamma$ 
of length $\delta$
and
for all sufficiently large $n$,
\begin{equation}
\label{TooShort} 
\mes(x\in \Gamma: r_n(x)\leq \delta_2)\leq \eps^2.
\end{equation}
Applying \eqref{LeaveLoc} with the assumption \eqref{ZSuperLarge} and with $\delta$ replaced by $\delta_2$ and
$\eps$ replaced
by $\delta_2 \eps$, we find that there exists
 $\tR$ so that for any curve $\Gamma$ of length greater than $\delta_2$ such that
$|z(\Gamma)|\geq \tR$ we have 
\begin{equation}
\label{FarThanClose}
 \mes(x\in \Gamma: z(\tT^n x)\leq R)\leq \eps^2 |\Gamma| \quad
\text{for}\quad n\geq n_0(\tR, \eps, \delta_2).
\end{equation}

Next for each $\Gamma$ with $|\Gamma|\geq \delta$, Corollary~\ref{CrNonLocalizeLoc} shows that there is some time
$n_1=n_1(\Gamma, \eps)$ such that 
\begin{equation}
\label{TooSlow}
\mes(x\in\Gamma: |z(\tT^{n_1}x)|\leq \tR)\leq \eps^2.
\end{equation}
By compactness there exists $N_1$ such that for all curves $\Gamma$ of length at least 
$\delta$ one has
$n_1(\Gamma, \eps)\leq N_1.$ 
Further increasing $N_1$ if necessary, we can assume that 
\eqref{TooShort} holds with $n=N_1$.
Next, take
$n\geq N_1+n_0(\tR, \eps, \delta_2).$ 
Divide the set of $x$ such that $|z(T^n x)|\leq R$ into three parts
$$ (i): r_{N_1}(x)\leq \delta_2, \quad (ii):
|z(\tT^{N_1}x)|\leq \tR, \quad $$
$$(iii): r_{N_1}(x)\geq \delta_2,  
|z(\tT^{N_1}x)|\geq \tR\text{ but } |z(\tT^n x)|\leq R .$$
Inequalities \eqref{TooShort}, \eqref{FarThanClose}, and \eqref{TooSlow} show 
 that contribution of each part to $\mes(x: |z(\tT^n x)|\leq R)$ is at most $\eps^2.$ This proves
\eqref{LeaveLoc} for 
$$n\geq N_1+n_0(\tR, \eps, \delta_2).$$

\section{Conclusions.}
\label{ScConclusion}
This paper deals with global mixing, that is, calculation of the expected value of an
extended observable in a long time limit, for mechanical systems.
The systems considered in this paper admit approximations at infinity,
that is, when either {the position or the velocity} is large, by a periodic system. It turns out that
if the map, obtained {from the approximating
system by factoring out the $\mathbb Z^d$ extension}, is chaotic (in our examples, the reduced {systems are hyperbolic systems}
with singularities), then the original system enjoys global global mixing.
To establish local global mixing, in addition to controlling
  the dynamics at infinity we also need to ensure the hyperbolicity in the whole phase space.
In particular, we gave examples, where local modifications of the dynamics destroy local global mixing.

We note that {notions of global mixing} discussed in this paper are neither implied by nor imply the classical properties
studied in infinite ergodic theory \cite{A97}. For example, Lorentz gas in a small external field is dissipative but
it enjoys both local global and global global mixing. 
Non mild local perturbations of Lorentz gas are conservative 
but not ergodic and they enjoy global global mixing (even though under natural assumptions, ergodicity
is a necessary prerequisite for local global mixing in the recurrent case, cf. discussion in \S \ref{SSSLocPert}). On the other hand,
{certain continuous time systems of 
bouncing balls in gravity field (i.e. special cases of the systems studied in \S \ref{SSGravity})} are likely to be ergodic and Krickeberg mixing but 
they are not global global mixing. This logical independence between global mixing and other infinite ergodic
theoretic properties is not surprising since those notions serve different purposes. Namely, classical ergodic
theory strives to control the ergodic sum of localized ($L^1$) observables and the notions such as Krickeberg mixing are useful for that purpose 
(see e.g. \cite{DSzV08, PS08,PT17}). The global mixing, on the other hand,
is useful for studying ergodic sums of extended observables (cf. \cite{BGL17, LM18}). In particular, 
it seems to us that
the global mixing is more suitable for derivation of macroscopic dynamics from microscopic laws,
as statistical mechanics concerns itself with extended observables.
In fact, in this paper we were able to prove
 
 (A) global global mixing for systems where a good control on the dynamics
in the bulk is already known and 

(B) local global mixing for systems where full limit theorems
are available due to a good
control of the boundary conditions ({ \cite{Ch01, C08}},  \cite{DSzV09, DN16}).

We also note that for mechanical systems there are more examples where the local global mixing is known
than the examples where the Krickeberg mixing was proven. Intuitively, proving local global mixing is 
easier since it only requires control on most of the phase space, while Krickeberg mixing requires
a good understanding of the dynamics in the localized regions of the phase space.

In summary global mixing is an interesting recent concept, which is relevant in 
{ several areas of mathematics including mathematical physics (cf. \cite{Kh49}),
dynamical systems (\cite{DDKN}), homogenization (\cite{DLN19})
and probability (\cite{DG})}
 and is easier to
establish than several other mixing properties. Our paper is a first step in studying global mixing for mechanical systems.
A natural next question to study is the Birkhoff theorem for global observables.
In \cite{DLN19} we address this question
  in the simplest setting, namely for i.i.d. random walks. However, since the main tool in \cite{DLN19} is the
  local limit theorem and related asymptotic expansions, we hope that the results similar to \cite{DLN19} also hold
  for many of the mechanical systems addressed here.

We also hope our work will stimulate further research on global mixing.
Some of the natural questions motivated by our results
include the multiple mixing, limit theorems for ergodic sums of global observables as well as 
quantitative aspects of global mixing.


\end{document}